\documentclass[12pt]{amsart}

\usepackage{tikz-cd}

\usepackage{amssymb}
\usepackage{amsthm}
\usepackage{amscd}

\usepackage[mathscr]{eucal}
\usepackage{enumitem}

\usepackage{hyperref}

\usepackage{color}

\newtheorem{Theorem}{Theorem}[section]
\newtheorem{theorem}[Theorem]{Theorem}

\newtheorem{proposition}[Theorem]{Proposition}

\newtheorem{lemma}[Theorem]{Lemma}

\newtheorem{fact}[Theorem]{Fact}
\newtheorem{remark/def}[Theorem]{Remark/Definition}

\newtheorem{claim}[Theorem]{Claim}

\theoremstyle{definition}

\newtheorem{example}[Theorem]{Example}

\newtheorem{remark}[Theorem]{Remark}

\newtheorem{definition}[Theorem]{Definition}

\newtheorem{notation}[Theorem]{Notation}

\newtheorem{question}[Theorem]{Question}
\newtheorem{def/rem}[Theorem]{Definition/Remark}
\newtheorem{def/fact}[Theorem]{Definition/Fact}
\newtheorem{not/rem}[Theorem]{Notation/Remark}
\newsavebox{\indbin}
\savebox{\indbin}{\begin{picture}(0,0)
\newlength{\gnu}
\settowidth{\gnu}{$\smile$} \setlength{\unitlength}{.5\gnu}
\put(-1,-.65){$\smile$} \put(-.25,.1){$|$}
\end{picture}}
\def \indo {\mathop{\smile \hskip -0.9em ^| \ }}
\def \depo {\mathop{ \not \smile \hskip -0.9em  ^| \ }}

\newcommand{\be}{\begin{enumerate}}
\newcommand{\bi}{\begin{itemize}}

\newcommand{\bd}{\begin{defn}}

\newcommand{\bt}{\begin{theorem}}
\newcommand{\bl}{\begin{lemma}}
\newcommand{\ee}{\end{enumerate}}
\newcommand{\ei}{\end{itemize}}
\newcommand{\ed}{\end{defn}}
\newcommand{\et}{\end{theorem}}
\newcommand{\el}{\end{lemma}}

\newcommand{\CP}{{\mathcal P}}

\newcommand{\CB}{{\mathcal B}}
\newcommand{\CC}{{\mathcal C}}

\newcommand{\CG}{{\mathcal G}}
\newcommand{\CL}{{\mathcal L}}
\newcommand{\CH}{{\mathcal H}}

\newcommand{\CJ}{\mathcal J}

\newcommand{\CN}{{\mathcal N}}

\newcommand{\BN}{{\mathbb N}}
\newcommand{\BZ}{{\mathbb Z}}

\newcommand{\CS}{\mathcal S}

\newcommand{\id}{\operatorname{id}}

\newcommand{\aut}{\operatorname{Aut}}

\newcommand{\Th}{\operatorname{Th}}

\def\res{\operatorname{res}}

\def\dcl{\operatorname{dcl}}

\def\acl{\operatorname{acl}}

\def\ker{\operatorname{Ker}}

\def\Im{\operatorname{Im}}

\def\tp{\operatorname{tp}}

\def\Sym{\operatorname{Sym}}

\def\ob{\operatorname{Ob}}
\def\mor{\operatorname{Mor}}

\def\qftp{\operatorname{qftp}}

\def\Sym{\operatorname{Sym}}
\def\SPG{\operatorname{SPG}}
\def\PG{\operatorname{PG}}
\def\Emb{\operatorname{Emb}}

\def\IM{\operatorname{Im}}
\def\SIM{\operatorname{SIm}}
\def\FSIM{\operatorname{FSIm}}
\def\coFSIM{\operatorname{coFSIm}}

\def\Pro{\operatorname{Pro}}

\keywords{sorted profinite groups, sorted complete system, sorted embedding property, co-sorted embedding property, sorted embedding cover}

\title{The embedding property for sorted profinite groups}
\author{Junguk Lee}
\address{Department of Mathematics, Changwon National University\\ Changwon 51140\\ South Korea}
\email{ljwhayo@changwon.ac.kr}

\begin{document}

\begin{abstract}
We study the embedding property in the category of sorted profinite groups. We introduce a notion of the sorted embedding property (SEP), analogous to the embedding property for profinite groups. We show that any sorted profinite group has a universal SEP-cover. Our proof gives an alternative proof for the existence of a universal embedding cover of a profinite group. Also our proof works for any full subcategory of the sorted profinite groups, which is closed under taking finite quotients, fibre products, and inverse limits.

We also show that any sorted profinite group having SEP has a sorted complete system whose theory is $\omega$-categorical and $\omega$-stable under the assumption that the set of sorts is countable.
\end{abstract}

\maketitle

\section{Introduction}\label{Section:intro}
For a profinite group $G$, let $\IM(G)$ be the set of isomorphism classes of finite quotients of $G$. We say that $G$ has the {\it embedding property} (EP) if for $A,B\in \IM(G)$ and for every epimorphisms $\Pi:A\rightarrow B$ and $\varphi:G\rightarrow B$, there is an epimorphism $\psi:G\rightarrow A$ such that $\Pi\circ\psi=\varphi$.

In field arithmetic and model theory of fields, the embedding property for profinite groups appears surprisingly. Let $k^{ab}$ and $k^{sol}$ be the maximal abelian extension and the maximal solvable extension of a number field $k$ respectively. In \cite{Iwa}, Iwasawa showed that the Galois group $G(k^{sol}/k^{ab})$ has the embedding property. A {\it Frobenius field} is a PAC field whose absolute Galois group has the embedding property. Over all Frobenius fields whose absolute Galois groups have the same set of isomorphism classes of finite quotients and which contain a common subfield $K$, Fried, Haran, and Jarden in \cite{FHJ} developed the theory of Galois stratification of $K$-constructible sets and proved an elimination of quantifiers of Galois formulas through Galois stratification. In \cite{HL}, Haran and Lubotzky gave a primitive recursive procedure to construct the universal embedding cover of a given finite group. Combined with the elimination of quantifiers of Galois formulas and the universal Frattini cover, they showed that the theory of perfect Frobenius fields is primitive recursive, and the theory of all Frobenius fields is decidable. 

Meanwhile, in model theory of PAC fields, one of most influential results is the work of Cherlin-van den Dries-Macintyre in \cite{CvDM}. In \cite{CvDM}, they first introduced a notion of a {\em complete system} of a profinite group, which encodes the inverse system of finite quotients of the profinite group by its open normal subgroups, and so the category of complete systems whose morphisms are embeddings and the category of profinite groups whose morphisms are epimorphisms are equivalent by a contravariant functor.  By the Elementary Equivalence Theorem for PAC fields (cf. \cite[Theorem 3.2]{JK75}), the theory of a PAC field is determined by its basic algebraic properties, for example, characteristic, imperfection degree, and the relative algebraic closure of the prime field, and the theory of the complete system of the Galois group of the field. So, model theoretic properties of PAC fields are reduced to model theoretic properties of complete systems. Also, the complete system is the `right' object to study profinite groups model theoretically. Most of all, the class of complete systems is an elementary class, and for the Galois group $G(K)$ of a field $K$, its complete system is interpretable in the pair consisting of the algebraic closure of $K$ and $K$. Note that an ultraproduct of profinite groups need not be a profinite group. In \cite{C1}, Chatzidakis showed that the complete system of a profinite group having the embedding property is $\omega$-stable. Using this with Chatzidakis' independence theorem in \cite[Theorem 3.1]{C2}, Ramsey in \cite[Theorem 3.9.31]{Ra} showed that the theory of a Frobenius field is NSOP$_1$.\\

In this article, we consider the embedding property in the category of sorted profinite groups. In \cite{HoLee}, we introduced a notion of sorted profinite groups to study the (Shelah-)Galois groups of first order structures. The Galois groups of first order structures are typical examples of sorted profinite groups (see Example \ref{ex:sorted_profinite_group}). In \cite[Proposition 5.6]{HoLee}, Hoffmann and the author developed the independence theorem for PAC structures, analogous to \cite[Theorem 3.1]{C2}, and we proved that for a PAC structure $M$, if the sorted complete system of the Galois group of $M$ is $\omega$-stable, then the theory of $M$ is NSOP$_1$. This leads us to try to find a class of sorted profinite groups having $\omega$-stable sorted complete systems. The possible candidates are the sorted profinite groups with ``the embedding property". So, we introduce a notion of the {\em sorted embedding property} (SEP) for a sorted profinite group, exactly analogous to the embedding property for profinite groups. Also, we need to find a notion of ``embedding property" which is first order axiomatizable. For this, we introduce a weaker notion of the finitely sorted embedding property (FSEP) (see Definition \ref{def:Iwasawa_property}). Fortunately, the two notions of the embedding properties for sorted profinite groups are equivalent (see Theorem \ref{thm:FSIP=SIP=IP+homogeneous_sorting_data}).

  We have two main results in this article. First, we show the existence and the uniqueness of the universal sorted embedding cover for a sorted profinite group (see Theorem \ref{thm:existence_universal_SIP_cover} and Theorem \ref{thm:uniqueness_SEP-cover}). Second, we show that the theory of the complete system of a sorted profinite group having the sorted embedding property is $\omega$-stable (see Theorem \ref{thm:description_complete_types}(2)).\\

In Section \ref{Section:preliminaries}, we introduce the category of sorted profinite groups whose morphisms are always epimorphisms, and we see that the category of sorted profinite groups is closed under the inverse limit and the fibre product. And we recall the category of sorted complete systems, which is equivalent to the category of sorted profinite groups via natural contravariant functors. In Section \ref{Section:universal_SIP_cover}, we prove the existence and uniqueness of a universal SEP-cover of  given sorted profinite groups. In Section \ref{Section:model theory of sorted profinite group with SIP}, we prove the theory of the sorted complete system of a sorted profinite group is $\omega$-stable under the assumption that the set of sorts is countable, and we describe the forking independence there.

\section{Preliminaries}\label{Section:preliminaries}
\subsection{Sorted profinite group}
For a profinite group $G$, we write $\CN(G)$ for the set of open normal subgroups of $G$. Let $\BN$ be the set of positive integers. {\bf Fix a set $\CJ$, called a set of {\em sorts}}

\begin{notation}
\begin{enumerate}
	\item For $n\in \BN$, let $\CJ^n$ be the set of $n$-tuples of elements in $\CJ$, and let $\CJ^{<\BN}:=\bigcup_{n\in \BN} \CJ^n$.
	\item For $J,J'\in \CJ^{<\BN}$, we write $J\le J'$ if $J$ is a subtuple of $J'$, not necessarily consisting of successive elements. 
	\item For $J,J'\in \CJ^{<\BN}$, we write $J^{\frown} J'$ for the concatenation of $J$ and $J'$.
	\item For $J=(j_1,\ldots,j_n)\in \CJ^n$, $|J|:=n$ and $\Vert J\Vert:=\{j_1,\ldots,j_n\}$.
	\item For $J=(j_1,\ldots,j_n)\in \CJ^n$ and a permutation $\sigma\in \Sym(n)$, $\sigma(J)=(j_{\sigma(1)},\ldots,j_{\sigma(n)})$.
	\item For $J\in \CJ^{<\BN}$, $\sqrt{J}:=\{J'\in \CJ^{<\BN}:\Vert J\Vert\subseteq \Vert J'\Vert\}$.
\end{enumerate}
\end{notation}

\noindent {\bf Fix two functions
\begin{itemize}
	\item $J_{\subseteq}^*:\BN\times \CJ^{<\BN}\rightarrow \CJ^{<\BN}$; and
	\item $J_{\cap}^*:\CJ^{<\BN}\times \CJ^{<\BN}\rightarrow \CJ^{<\BN},(J,J')\mapsto J^{\frown}J'$.
\end{itemize}}
\noindent Note the function $J_{\subseteq}^*$ can be arbitrary. The functions $J_{\subseteq}^*$ and $J_{\cap}^*$ are necessary to axiomatize the {\em sorted complete system} of a sorted profinite group in an appropriate first order language so that the sorted complete system satisfies the modular lattice axiom (cf. \cite[Definition 3.7]{HoLee}).

\begin{definition}\label{def:sorting_data}
For a profinite group $G$, we associate a non-empty subset $F(N)$ of $\CJ^{<\BN}$ for each $N\in \CN(G)$ and consider an indexed family $F:=\{F(N):N\in \CN(G)\}$. We say that the indexed family $F$ is a {\em sorting data} of $G$ if the following hold: For $N,N_1,N_2\in \CN(G)$,
\begin{enumerate}
	\item $F(G)=\CJ^{<\BN}$.
	\item $J\in F(N)\Leftrightarrow \sqrt{J}\subseteq F(N)$;
	\item Suppose $N_1\subseteq N_2$ and $[G:N_1]\le k$. For $J\in \CJ^{<\BN}$, $$J\in F(N_1)\Rightarrow J_{\subseteq}^*(k,J)\in F(N_2).$$
	\item For $J_1\in F(N_1)$ and $J_2\in F(N_2)$, $J_{\cap}^*(J_1,J_2)\in F(N_1\cap N_2)$.
\end{enumerate}
We call the pair $(G,F)$ a {\em sorted profinite group}, and we say that the sorting data $F$ comes from $\CJ$. For sorting data $F$, $F'$ on $G$, we write $F\subseteq F'$ if $F(N)\subseteq F'(N)$ for any $N\in \CN(G)$. A sorting data $F$ on $G$ is called {\em full} if $F(N)=\CJ^{<\BN}$ for all $N\in \CN(G)$. 
\end{definition}

A typical example of a sorted profinite group comes from the Galois group of a first order structure by attaching additional data to each open normal subgroup of the Galois group, which recognize on which sorts a primitive element of a finite Galois extension of the structure live (cf. \cite[Section 2]{HoLee}). So, `sorted' in a sorted profinite group means that each open normal subgroup is sorted in a certain way.

\begin{example}\label{ex:typical_example_galois_gp}
Fix a first order language $\CL$ with a set $\CJ$ of all sorts. Let $T$ be a complete $\CL$-theory eliminating quantifiers and imaginaries, and let $\mathfrak{C}\models T$ be a monster model. For each $J=(S_1,\ldots,S_k)\in \CJ^{<\BN}$, write $S_J$ for the sort of $S_1\times\cdots\times S_k$.

Let $K\subset \mathfrak{C}$ be a small definably closed substructure. Consider the Galois group of $K$, $$G(K):=G(\acl(K)/K)=\{\varphi\restriction_{\acl(K)}:\varphi\in\aut(\mathfrak{C}/K)\}.$$ For each $N\in \CN(G(K))$, let $\acl(K)^N$ be the substructure consisting of elements in $\acl(K)$ fixed pointwise by each $\sigma\in N$, that is, $$\acl(K)^N:=\{x\in \acl(K):\forall \sigma\in N\left(\sigma(x)=x\right)\}.$$ Note that there is a finite tuple $a\in \acl(K)^N$ such that $\acl(K)^N=\dcl(a,K)$ and we call such an element a {\em primitive element} of $\acl(K)^N$ over $K$ (cf. \cite[Fact 2.8]{HoLee}). An extension $K\subseteq L \subseteq \acl(K)$ is called a {\em Galois extension} of $K$ if $L=\dcl(L)$ and for each $\sigma\in \aut(\mathfrak{C}/K)$, $\sigma[L]=L$. For a Galois extension $K\subseteq L$, the Galois group of $L$ over $K$ is $$G(L/K):=\{\sigma\restriction_L:\sigma\in G(K)\}.$$

Now, we define a natural sorting data $F$ on $G(K)$ as follows: For $N\in \CN(G(K))$ with $[G(K):N]=n$ and for $J\in \CJ^{<\BN}$, $J\in F(N)$ if and only if there is a finite tuple $a$ in $(S_J(\acl(K)))^n$ such that $\dcl(K,a)=\acl(K)^{N}$. Note that there is a function $J_{\subseteq}^*:\BN\times \CJ^{<\BN}\rightarrow \CJ^{<\BN}$ defined in \cite[Remark 3.1]{HoLee} such that for $N\subseteq N'\in \CN(G(K))$ with $[G(K):N]\le k$ and for $J\in \CJ^{<\BN}$, if $J\in F(N)$, then $J_{\subseteq}^*(k,J)\in F(N')$. Similarly, for a Galois extension $L$ of $K$, we define a natural sorting data $F$ on $G(L/K)$ as follows: For $N\in \CN(G(L/K))$ with $[G(L/K):N]=n$ and for $J\in \CJ^{<\BN}$, $J\in F(N)$ if and only if there is a finite tuple $a$ in $(S_J(L))^n$ such that $\dcl(K,a)=L^N$.
\end{example}

\subsection{The category of sorted profinite groups}
The category $\PG$ of profinite groups is a category consisting of the following:
\begin{itemize}
	\item $\ob$ : The objects of $\PG$ are profinite groups.
	\item $\mor$ : Let $G_1$ and $G_2$ be profinite groups. A morphism from $G_1$ to $G_2$ is a continuous homomorphism from $G_1$ to $G_2$.
\end{itemize}

Next, we introduce the category of sorted profinite groups whose sorting data come from $\CJ$. 
\begin{definition}\label{def:category_sorted_profinite_group}
The category $\SPG_{\CJ}(J_{\subseteq}^*,J_{\cap}^*)$ of sorted profinite groups with sorting data from $\CJ$ consists of the following:
\begin{itemize}
	\item $\ob$ : The objects are sorted profinite groups whose sorting data comes from $\CJ$; and
	\item $\mor$ : Let $(G_1,F_1)$ and $(G_2, F_2)$ be in $\ob(\SPG_{\CJ}(J_{\subseteq}^*,J_{\cap}^*))$. A morphism $f$ from $(G_1,F_1)$ to $(G_2,F_2)$ is an {\bf epimorphism} from $G_1$ to $G_2$, that is, a surjective continuous homomorphism, satisfying that for $N\in \CN(G_2)$, $$F_2(N)\subseteq F_1(f^{-1}[N]).$$
\end{itemize}

For sorted profinite groups $(G,F)$ and $(G',F')$, we say that a epimorphism $\varphi:G'\rightarrow G$ is {\em sorted} with respect to $F'$ and $F$ if the epimorphism $\varphi$ induces a morphism from $(G',F')$ to $(G,F)$, that is, for each $N\in \CN(G)$, $$F(N)\subseteq F'(\varphi^{-1}[N]).$$ If the sorting data $F$ and $F'$ are clear from the context, we say $\varphi$ is sorted instead of being sorted with respect to $F'$ and $F$. If there is no confusion, we write $\SPG$ for $\SPG_{\CJ}(J_{\subseteq}^*,J_{\cap}^*)$.
\end{definition}

The concept of a sorted epimorphism first appeared in \cite[Section 5]{DHL20} to study an ultraproduct of isomorphisms between Galois groups of first order structures. Note that in the category of sorted profinite groups, every morphism is surjective. To consider only epimorphisms between sorted profinite groups is natural in the view that the category of sorted profinite groups and the category of sorted complete systems are equivalent via a contravariant functor (cf. Subsection \ref{Section:sorted_complete_system}).

\begin{example}\label{ex:sorted_profinite_group} Let us come back to Example \ref{ex:typical_example_galois_gp} for a moment. For small definably closed substructures $K\subseteq K'\subset \mathfrak{C}$, we say $K'$ is a {\em regular extension} of $K$ if $K'\cap \acl(K)=K$. Let $K\subset K' \subset \mathfrak{C}$ be small definably closed substructures such that $K'$ is a regular extension of $K$. Then, the restriction map $\res:G(K')\rightarrow G(K)$ induces a morphism from $(G(K'),F')$ to $(G(K),F)$, where $F$ and $F'$ are the natural sorting data on $G(K)$ and $G(K')$ respectively. Namely, for each $N\in \CN(G(K))$ and for a finite tuple $a$ in $\acl(K)^N$, $$\dcl(a,K)=\acl(K)^N\Leftrightarrow \dcl(a,K')=\acl(K')^{\res^{-1}[N]}$$ so that $F(N)\subseteq F'(\res^{-1}[N])$ (cf. \cite[Fact 2.10]{HoLee}).
\end{example}

\bigskip

Our main goal in this subsection is to show that the category $\SPG$ is closed under taking the inverse limit and taking the fibre product. We first introduce some terminology.

\begin{notation}\label{notation:union_intersection_presortingdata}
Let $I$ be an index set. For each $i\in I$, let $F_i:=\{F_i(N)\subseteq \CJ^{<\BN}:N\in \CN(G)\}$ be a $\CN(G)$-indexed set of subsets of $\CJ^{<\BN}$.
\begin{enumerate}
	\item Let $\bigcap_{i\in I}F_i$ be a $\CN(G)$-indexed family given as follows: For each $N\in \CN(G)$, $(\bigcap_{i\in I}F_i)(N):=\bigcap_{i\in I}F_i(N)$.
	\item Let $\bigcup_{i\in I}F_i$ be a $\CN(G)$-indexed family given as follows: For each $N\in \CN(G)$, $(\bigcup_{i\in I}F_i)(N):=\bigcup_{i\in I}F_i(N)$.
\end{enumerate}
\end{notation}

\begin{remark}\label{rem:minimal_sorting_data}
\begin{enumerate}
	\item For an indexed family $\hat F=\{\hat F(N)(\neq \emptyset):N\in \CN(G)\}$, there is a unique minimal sorting data $F$ on $G$ such that for each $N\in \CN(G)$, $F(N)$ contains $\hat F(N)$, which is called {\em generated by $\hat F$}. Note that such a minimal one is given by taking intersections of all sorting data on $G$ containing $\hat F$. We call an indexed family $\hat F$ with $\hat F(N)\neq \emptyset$ for each $N\in \CN(G)$ a {\em pre-sorting data on $G$}.
	\item Let $F_i=\{F_i(N):N\in \CN(G)\}$ be a sorting data of a profinite group $G$ for each $i\in I$. If $\bigcap_{i\in I}F_i(N)\neq \emptyset$ for each $N\in \CN(G)$, both of the pre-sorting data $\bigcap F_i$ and $\bigcup_{i\in I} F_i$ are sorting data on $G$. If $I$ is finite, then $\bigcap_{i\in I}F_i(N)\neq \emptyset$ for each $N\in \CN(G)$ so that $\bigcap_{i\in I} F_i$ is a sorting data on $G$. Namely, suppose $I=\{i_1,\ldots,i_n\}$. Fix $N\in \CN(G)$. For each $i_j$, choose $J_j\in F_{i_j}(N)$. Then, $J_1^{\frown}J_2^{\frown}\cdots^{\frown}J_n$ is in $F_{i_1}(N)\cap\cdots\cap F_{i_n}(N)$ by Definition \ref{def:sorting_data}(2).
\end{enumerate} 
\end{remark}

\begin{definition}\label{def/rem:pushforward_sortingdata}
Let $\varphi:G_1\rightarrow G_2$ be an epimorphism. Let $F_1$ be a sorting data on $G_1$. Consider a pre-sorting data $\hat F_2$ on $G_2$ given by $\hat F_2(N_2):=F_1(f^{-1}[N_2])$ for $N_2\in \CN(G_2)$. Then, the pre-sorting data $\hat F_2$ is a sorting data, and we call this sorting data the {\em push-forward sorting data} of $F_1$ along $\varphi$, denoted by $\varphi_*(F_1)$. 
\end{definition}
\noindent Note that for any sorting data $F_2'$ on $G_2$, $\varphi:(G_1,F_1)\rightarrow (G_2,F_2')$ is sorted if and only if $F_2'\subseteq \varphi_*(F_1)$.

\begin{definition}\label{def:base_at_e}
Let $G$ be a profinite group and let $e$ be the identity of $G$.
\begin{enumerate}
	\item We say that a subset $\CB\subseteq \CN(G)$ is a {\em base at $e$} if for any $N\in \CN(G)$, there is $N'\in \CB$ such that $N'\subseteq N$.
	\item We say that a subset $X\subseteq \CN(G)$ {\em generates a base at $e$} if the set $\CB(X):=\{N_1\cap \cdots \cap N_k:N_i\in X\}$ forms a base at $e$, equivalently, $\bigcap X=\{e\}$. Indeed, if $X$ generates a base at $e$, then $\bigcap X=\{e\}$ because $G$ is a Hausdorff space. Conversely, suppose $\bigcap X=\{e\}$. Take $N\in \CN(G)$ arbitrary. Suppose $N_1\cap \cdots \cap N_k\not\subseteq N$ for any $N_1,\ldots,N_k\in X$. Then, by compactness, we have that $\bigcap X\cap G\setminus N\neq \emptyset$ and $e\not\in N$, which is a contradiction. So, for some $N_1,\ldots,N_k\in X$, $N_1\cap\cdots\cap N_k\subseteq N$. 
\end{enumerate}
\end{definition}

\begin{remark}\label{def/rem:generating_sording_data}
Let $G$ be a profinite group.
\begin{enumerate}
	\item Let $\CB\subseteq \CN(G)$ be a base at $e$. For each $N\in \CB$, choose $F_{\CB}(N)(\neq \emptyset)\subseteq \CJ^{<\BN}$, and put $F_{\CB}:=\{F_{\CB}(N):N\in \CB\}$, called a {\em pre-sorting data on $\CB$}. Then, there is a unique minimal sorting data $F$ such that for each $N\in \CB$, $F(N)$ contains $F_{\CB}(N)$. In this case, we say that $F$ is generated by $F_{\CB}$.
	\item Let $X\subseteq \CN(G)$ generate a base at $e$. For each $N\in X$, choose $F_X(N)(\neq \emptyset)\subseteq \CJ^{<\BN}$, and put $F_{X}:=\{F_X(N):N\in X\}$, called a {\em pre-sorting data on $X$}. Then, there is a unique minimal sorting data $F$ such that for each $N\in X$, $F(N)$ contains $F_X(N)$. In this case, we say that $F$ is generated by $F_{X}$.
\end{enumerate}
\end{remark}
\begin{proof}
$(1)$ Let $\CB\subseteq \CN(G)$ be a base at $e$ and let $F_{\CB}$ be a pre-sorting data on $\CB$. Define a pre-sorting data $\hat F$ on $G$ given as follows: For $N\in \CN(G)$,
\begin{itemize}
	\item if $N\in \CB$, put $\hat F(N):=F_{\CB}(N)$; and
	\item if $N\not\in \CB$, put $\hat F(N):=\bigcup_{N'\subseteq N}\bigcup_{k\ge [G:N']}J_{\subseteq }^*[\{k\}\times F_{\CB}(N') ]$.
\end{itemize}
Let $F$ be a sorting data of $G$ generated by $\hat F$, which exists by Remark \ref{rem:minimal_sorting_data}. Then, the sorting data $F$ is also generated by $F_{\CB}$.\\

$(2)$ Let $X\subseteq \CN(G)$ generate a base at $e$. Put $\CB:=\{N_1\cap \cdots \cap N_k:N_i\in X\}$, which is a base at $e$. Let $F_X$ be a pre-sorting data on $X$. Define a pre-sorting data $F_{\CB}$ on $\CB$ given as follows: For $N\in \CB$, put $$F_{\CB}(N):=\bigcup_{N_1,\ldots,N_k\in X, N=N_1\cap \cdots \cap N_k}\limits \{J_1^{\frown}\cdots^{\frown}J_k:J_1\in F_{X}(N_1),\ldots, J_k\in F_X(N_k)\}.$$ Let $F$ be a sorting data of $G$ generated by $F_{\CB}$, which exists by $(1)$. Then, the sorting data $F$ is also generated by $F_{X}$.
\end{proof}

\bigskip

Now, we show that the category $\SPG$ is closed under taking the inverse limit.
\begin{proposition}\label{prop:inverselimit_sorted_prof_gp}
The category $\SPG$ is closed under taking the inverse limit.
\end{proposition}
\begin{proof}
Consider an inverse system $((G_i,F_i),f_j^i:(G_i,F_i)\rightarrow (G_j,F_j))_{j\le i\in I}$ of sorted profinite groups indexed by a directed poset $(I,\le)$. Let $G$ be the inverse limit of $G_i$ in the category of profinite groups, which is a profinite group. Then, for each $i\in I$, there is an epimorphism $f_i:G\rightarrow G_i$ such that for $j\le i$, $f_j=f_j^i\circ f_i$.

We consider a pre-sorting data $\hat F$ on $G$ given as follows: Let $N\in \CN(G)$. Put $I_N:=\{i\in I:N=f_i^{-1}[f_i[N]]\}$, equivalently, $i\in I_N$ if and only if $\ker f_i\subseteq N$. Note that $I_N\neq \emptyset$. Put $$\hat F(N):=\bigcup_{i\in I_N} F_i(f_i[N]).$$ Note that $F_j(f_j[N])\subseteq F_i(f_i[N])$ for $j\le i \in I_N$.
\begin{claim}\label{claim:sorting_data_inverselimit}
The pre-sorting data $\hat F$ is a sorting data.
\end{claim}
\begin{proof}
It is enough to show that $\hat F$ satisfies the conditions $(3)$ and $(4)$ in Definition \ref{def:sorting_data}. We first show that the condition $(3)$ holds for $\hat F$. Take $N_1\subseteq N_2\in \CN(G)$ and $k\in \BN$ with $[G:N_1]\le k$. Take $J\in \CJ^{<\BN}$ with $J\in \hat F(N_1)$. Then, by definition, there is $i\in I$ such that $\ker f_i\subseteq N_1$ and $J\in F_i(f_i[N_1])$. We have that $J_{\subseteq}^*(k,J)\in F_i(f_i[N_2])$ because $F_i$ is a sorting data of $G_i$ and $f_i[N_1]\subseteq f_i[N_2]\in \CN(G_i)$. Since $\ker f_i\subseteq N_1\subseteq N_2$, we have that $i\in I_{N_2}$ and $J_{\subseteq}^*(k,J)\in F_i(f_i[N_2])\subseteq \hat F(N_2)$.

Next, we show that the condition $(4)$ holds. Take $J_1\in \hat F(N_1)$ and $J_2\in \hat F(N_2)$. Take $i\in I$ such that
\begin{itemize}
	\item  $\ker f_i\subseteq N_1\cap N_2$; and
	\item $J_1\in F_i(f_i[N_1])$ and $J_2\in F_i(f_i[N_2])$.
\end{itemize}
We have that $f_i[N_1]\cap f_i[N_2]=f_i[N_1\cap N_2]$ because $\ker f_i\subseteq N_1\cap N_2$. Also, since $F_i$ is a sorting data of $G_i$, we have that $$J_{\cap}^*(J_1,J_2)\in F_i(f_i[N_1]\cap f_i[N_2])=F_i(f_i[N_1\cap N_2]).$$
\end{proof}

\begin{claim}\label{claim:inverselimit}
The sorted profinite group $(G,\hat F)$ satisfies the following universal property: Let $(G',F')$ be a sorted profinite group and let $g_i :(G',F')\rightarrow (G_i,F_i)$ be a morphism for each $i\in I$. Then, there is a morphism $g:(G',F')\rightarrow (G,F)$ such that for each $i$, $g_i=f_i\circ g$. 
\end{claim}
\begin{proof}
Since $G$ is the inverse limit of $G_i$ in the category of profinite groups, there is a morphism $g:G'\rightarrow G$ such that for each $i$, $g_i=f_i\circ g$. It is enough to show that $g$ is a morphism in the category of $\SPG$, that is, for each $N\in \CN(G)$, $$\hat F(N)\subseteq F'(g^{-1}[N]).$$ Take $J\in \hat F(N)$. By definition, there is $i\in I$ such that $\ker f_i\subseteq N$ and $J\in F_i(f_i[N])$. Since $\ker f_i\subseteq N$ and $g_i=f_i\circ g$, we have that $$g^{-1}[N]=g^{-1}[f_i^{-1}[f_i[N]]]=g_i^{-1}[f_i[N]].$$ Since $g_i$ is a morphism in $\SPG$ and $J\in F_i(f_i[N])$, we have that  $$J\in F'(g_i^{-1}[f_i[N]])=F'(g^{-1}[N]).$$
\end{proof}
\noindent By Claim \ref{claim:inverselimit}, the sorted profinite group $(G,\hat F)$ is the inverse limit in the category $\SPG$.
\end{proof}

\bigskip

Next, we consider a notion of the fibre product in the category $\SPG$. 
\begin{definition}\label{def:fibre_product}
Let $\Pi_1:(B_1,F_1)\rightarrow (A,F_A)$ and $\Pi_2:(B_2,F_2)\rightarrow (A,F_A)$ be morphisms of sorted profinite groups so that they are epimorphisms. The {\em fibre product} of $B_1$ and $B_2$ over $A$ with respect to $\Pi_1$ and $\Pi_2$ is the following sorted profinite group $(B,F)$:
\begin{itemize}
	\item $B=B_1\times_A B_2=\{(b_1,b_2)\in B_1\times B_2:\Pi_1(b_1)=\Pi_2(b_2)\}$, which is the fibre product in the category of profinite groups; and
	\item Let $p_1:B\rightarrow B_1$ and $p_2:B\rightarrow B_2$ be the projection maps. Let $X=\{p_1^{-1}[N_1]:N_1\in \CN(B_1)\}\cup\{p_2^{-1}[N_2]:N_2\in \CN(B_2)\}$, which generates a base at $e$ because $\bigcap X=\{(e_1,e_2)\}$ where $e_1$ and $e_2$ are the identities of $B_1$ and $B_2$ respectively. Let $F_X$ be a pre-sorting data on $X$ given as follows: For $N_1\in \CN(G_1)$ and $N_2\in \CN(G_2)$, $$F_X(p_1^{-1}[N_1])=F_1(N_1),\ F_X(p_2^{-1}[N_2])=F_2(N_2).$$ Let $F$ be the sorting data generated by $F_X$.
\end{itemize}
The sorting data $F$ is the minimal one among sorting data $F'$ on $B$ to make $p_1$ and $p_2$ morphisms in the category $\SPG$, that is, the sorting data $F$ makes the projection maps $p_1$ and $p_2$ sorted and for any such sorting data $F'$, $F\subseteq F'$. 
\end{definition}
In the category of profinite groups, the fibre product is characterized by the following properties:
\begin{remark}\label{rem:char_fired_prod}\cite[Lemma 1.1]{HL}
Consider a commutative diagram of groups with epimorphisms  in the category of sorted profinite groups: 
$$
\begin{tikzcd}
B \arrow[d, "p_1"'] \arrow[r, "p_2"] & B_2 \arrow[d, "\Pi_2"] \\
B_1 \arrow[r, "\Pi_1"']& A
\end{tikzcd}
$$
and put $p=\Pi_1\circ p_1=\Pi_2\circ p_2$. The following are equivalent:
\begin{enumerate}
	\item $B$ is isomorphic to the fibre product of $B_1$ and $B_2$ over $A$.
	\item $B$ with $p_1$ and $p_2$ is a pullback of the pair $(\Pi_1,\Pi_2)$, that is, for any morphisms $\psi_i:C\rightarrow B_i$ for $i=1,2$ with $\Pi_1\circ \psi_1=\Pi_2\circ \psi_2$, there is a unique morphism $\psi:C\rightarrow B$ such that $p_i\circ \psi=\psi_i$ for $i=1,2$.
$$
\begin{tikzcd}
C
\arrow[dr, "\psi"]
\arrow[drr, bend left, "\psi_2"]
\arrow[ddr, bend right, "\psi_1"'] & & \\
& B \arrow[d, "p_1"'] \arrow[r, "p_2"] & B_2 \arrow[d, "\Pi_2"]\\
& B_1 \arrow[r, "\Pi_1"'] & A
\end{tikzcd}
$$
	\item $\ker p_1\cap \ker p_2=\{e\}$, and $A$ with $\Pi_1, \Pi_2$ is a  pushout of the pair $(p_1,p_2)$, that is, for any morphism $\varphi_i:B_i\rightarrow G$ for $i=1,2$ with $\varphi_1\circ p_1=\varphi_2\circ p_2$, there is a unique morphism $\varphi:A\rightarrow G$ such that $\varphi \circ \Pi_i=\varphi_i$ for $i=1,2$.
$$
\begin{tikzcd}
B \arrow[d, "p_1"'] \arrow[r, "p_2"] & B_2 \arrow[d, "\Pi_2"] \arrow[ddr, bend left, "\varphi_2"] & \\
B_1 \arrow[r, "\Pi_1"'] \arrow[drr, bend right, "\varphi_1"'] & A \arrow[dr, "\varphi"] & \\
& & G
\end{tikzcd}
$$
	\item $\ker p=\ker p_1 \times \ker p_2$.
\end{enumerate}
Note the universal property of $(2)$ does not make sense in the category $\SPG$ because the map $\psi$ need not be surjective even though both $\psi_1$ and $\psi_2$ are surjective.
\end{remark}

We borrow a notion of a cartesian diagram in the category of profinite groups from \cite[p. 185]{HL}.
\begin{definition}\label{def:cartesian_diagram}
We say that a diagram of sorted profinite groups in the category $\SPG$
$$
\begin{tikzcd}
(B,F) \arrow[d, "p_1"'] \arrow[r, "p_2"] & (B_2,F_2) \arrow[d, "\Pi_2"] \\
(B_1,F_1) \arrow[r, "\Pi_1"']& (A, F_A)
\end{tikzcd}
$$
is called {\em cartesian} if $(B,F)$ is isomorphic to the fibre product of $(B_1,F_1)$ and $(B_2,F_2)$ over $(A,F_A)$.
\end{definition}
\noindent We have the following analogue to \cite[Lemma 1.2]{HL}.

\begin{remark}\label{rem:inducing_fibre_prod}
Let $\psi_i:(C,F_C)\rightarrow (B_i,F_i)$ be a morphism for $i=1,2$. Then there is a commutative diagram:
$$
\begin{tikzcd}
(C,F_C)
\arrow[dr, "\psi"]
\arrow[drr, bend left, "\psi_2"]
\arrow[ddr, bend right, "\psi_1"'] & & \\
& (B,F) \arrow[d, "p_1"'] \arrow[r, "p_2"] & (B_2,F_2) \arrow[d, "\Pi_2"]\\
& (B_1,F_1) \arrow[r, "\Pi_1"'] & (A,F_A)
\end{tikzcd}
$$
where the square is cartesian and $\psi$ is a morphism.
\end{remark}
\begin{proof}
By \cite[Lemma 1.2]{HL}, there is a unique commutative digram 
$$
\begin{tikzcd}
C
\arrow[dr, "\psi"]
\arrow[drr, bend left, "\psi_2"]
\arrow[ddr, bend right, "\psi_1"'] & & \\
& B \arrow[d, "p_1"'] \arrow[r, "p_2"] & B_2 \arrow[d, "\Pi_2"]\\
& B_1 \arrow[r, "\Pi_1"'] & A
\end{tikzcd}
$$
such that $B$ is the fibre product of $B_1$ and $B_2$ over $A$ with an epimorphism $\psi$ in the category of profinite groups. Let $F_A$ be an arbitrary sorting data on $A$ such that the epimorphism $\Pi_i:(B_i,F_i)\rightarrow (A,F_A)$ is sorted for each $i=1,2$. Take the sorting data $F$ such that $(B,F)$ is the fibre product of $(B_1,F_1)$ and $(B_2,F_2)$ over $(A,F_A)$ so that the square is cartesian in the category $\SPG$. Note that the sorting data $F$ depends only on two morphisms $p_1$ and $p_2$.

It remains to show that the surjective homomorphism $\psi:(C,F_C)\rightarrow (B,F)$ is sorted. Put $X=\{p_1^{-1}[N_1]:N_1\in\CN(B_1)\}\cup\{p_2^{-1}[N_2]:N_2\in \CN(B_2)\}$. Take $N_1\in \CN(B_1)$ and $J_1\in F_1(N_1)$. Then, we have that $$J_1\in F_C(\psi_1^{-1}[N_1])=F_C(\psi^{-1}[p_1^{-1}[N_1]])$$ because $\psi_1=p_1\circ \psi$. The same things holds for $N_2\in \CN(B_2)$ and $J_2\in F_2(N_2)$. Thus, by the minimality of $F$, we have that $F[N]\subseteq F_C(\psi^{-1}[N])(=\psi_*(F_C)[N])$ for any $N\in X$ and so $F\subseteq \psi_*(F_C)$. So, by Remark \ref{def/rem:pushforward_sortingdata}(1), the epimorphism $\psi:(C,F_C)\rightarrow (B,F)$ is sorted.
\end{proof}

\subsection{Sorted embedding property}
We introduce the {\em sorted embedding property} for sorted profinite groups, analogous to the embedding property for profinite groups in \cite{HL} (or also called the Iwasawa property in \cite{C1}). We start with the definition of {\em embedding condition} from \cite[p. 185]{HL}. Fix a set $\CJ$. Throughout this section, a sorted profinite group means a sorted profinite group in $\SPG$. Let $(G,F)$ be a sorted profinite group. For a pair $((A,F_A),(B,F_B))$ of sorted profinite groups, the {\em sorted embedding condition}, denoted by $\Emb_{(G,F)}((A,F_A),(B,F_B))$, is defined as follows: If $(A,F_A)$ is a quotient of $(G,F)$, then for every pair of morphisms $\Pi:(A,F_A)\rightarrow (B,F_B)$ and $\varphi:(G,F)\rightarrow (B,F_B)$, there is a morphism $\psi:(G,F)\rightarrow (A,F_A)$ such that $\Pi\circ \psi=\varphi$. Let $\SIM(G,F)$ be the set of isomorphism classes of sorted finite quotients of $(G,F)$. 
\begin{definition}\label{def:Iwasawa_property}
We say that a sorted profinite group $(G,F)$ satisfies the {\em sorted embedding property} (SEP) if for all $(A,F_A),(B,F_B)\in \SIM(G,F)$, the condition $\Emb_{(G,F)}((A,F_A),(B,F_B))$ holds.
\end{definition}

Also, we introduce the {\em finitely sorted embedding property} (FSEP), which is weaker than SEP, but they will turn out to be equivalent in Theorem \ref{thm:FSIP=SIP=IP+homogeneous_sorting_data}. The advantage of FSEP is that it can be first order axiomatizable in the language of sorted complete systems. We say that a sorting data $F$ on a profinite group $G$ is {\em finitely generated} if there are a subset $X$ of $\CN(G)$ generating a base at $e$ and a pre-sorting data $F_{X}$ such that
	\begin{itemize}
		\item $F_X$ generates $F$; and
		\item $F_X(N)$ is finite for each $N\in X$.
	\end{itemize}
In this case, we say that $(G,F)$ is {\em finitely sorted}. Let $\FSIM(G,F)$ be the set of isomorphism classes of finitely sorted finite quotients of $(G,F)$. Clearly we have that $\FSIM(G,F)\subset \SIM(G,F)$.

\begin{definition}\label{def:fin_sorted_Iwasawa_property}
We say that a sorted profinite group $(G,F)$ satisfies the {\em finitely sorted embedding property} (FSEP) if for all $(A,F_A),(B,F_B)\in \FSIM(G,F)$, the condition $\Emb_{(G,F)}((A,F_A),(B,F_B))$ holds.
\end{definition}

\begin{example}\label{ex:sorted_profinite_group_having_SEP}
\begin{enumerate}
	\item If $G$ is a profinite group having the embedding property, then the sorted profinite group $(G,F)$ has SEP where $F$ is the full sorting data on $G$. For example, the free profinite group has the embedding property.
	
In general, for any profinite group $H$, there is an epimorphism $p:G\rightarrow H$ such that $G$ has the embedding property. In this case, we call such a $G$ an {\em embedding cover}. We say an embedding cover $G$ of $H$ is {\em universal} if for any embedding cover $G'$ of $H$ and any epimorphisms $p:G\rightarrow H$ and $p':G'\rightarrow H$, there is an epimorphism $q:G'\rightarrow G$ such that $p'=p\circ q$. Moreover, any profinite group has a unique universal embedding cover and any finite group has a unique finite universal embedding cover (see \cite[Theorem 2.7]{C1} and \cite[Theorem 1.12]{HL}).
	\item Let $\CJ=\{s_1,s_2\}$. Put $J_{\subseteq}^*:\BN\times \CJ^{<\BN}\rightarrow \CJ^{<\BN}, (k,J)\mapsto J^{\frown}(s_1,s_2)$. Let $G=\BZ/2\BZ\times \BZ/2\BZ$, which has the embedding property as a profinite group. Then $$\CN(G)=\{0, G,N_{(1,1)}, N_{(1,0)}, N_{(0,1)} \}$$ where $N_a$ is the subgroup of $G$ generated by $a$ for $a\in G$. Define a sorting data $F$ on $G$ as follows:
	\begin{itemize}
		\item $F(G)=\CJ^{<\BN}$;
		\item $F(N_{(1,0)})=F(N_{(0,1)})=F(N_{(1,1)})=\{J\in \CJ^{<\BN}:s_1\in ||J||\}$;
		\item $F(0)=\{J\in \CJ^{<\BN}:s_1,s_2\in ||J||\}$.
	\end{itemize}
Then, $(G,F)$ has SEP.	
\end{enumerate}
  
\end{example}

Now, we show that the weaker notion of FSEP is actually equivalent to the notion of SEP.

\begin{lemma}\label{lem:descritpion_pushforward_sortingdata_for_isoimage}
Let $(G,F)$ be a sorted profinite group having FSEP and let $N,N'\in \CN(G)$ with $G/N\cong G/N'$. Then, $F(N)=F(N')$
\end{lemma}
\begin{proof}
Take $N,N'\in \CN(G)$ with $\varphi:G/N'\cong G/N$. For a contradiction, suppose $F(N)\neq F(N')$. Without loss of generality, we may assume that there is $J'\in F(N')\setminus F(N)$. Let $A:=G/N$, $A':=G/N'$, and let $\pi:G\rightarrow A$ and $\pi':G\rightarrow A'$ be the  projections from $G$ to $A$ and $A'$. For each $\tilde N\in \CN(A)$, choose $$\tilde J\in F(\pi^{-1}[\tilde N])\cap F((\pi')^{-1}\circ \varphi^{-1}[\tilde N]).$$ Consider a pre-sorting data on $A$ and $A'$ given as follows: For $\tilde N\in \CN(A)$,
\begin{align*}
\hat F_A(\tilde N)&:=\{\tilde J\},\\
\hat F_{A'}(\varphi^{-1}[\tilde N])&:=\begin{cases}
\{\tilde J\}&\varphi^{-1}[\tilde N]\neq \{e_{A'}\},\\
\{\tilde J,J'\}&\varphi^{-1}[\tilde N]= \{e_{A'}\}.
\end{cases}
\end{align*}
Let $F_A$ and $F_{A'}$ be the sorting data generated by $\hat F_A$ and $\hat F_{A'}$ respectively. Then, clearly $(A,F_A),(A',F_{A'})\in \FSIM(G,F)$. Since $(G,F)$ has FSEP, we have the following diagram: 
$$
\begin{tikzcd}
& (G,F) \arrow[dl, dashrightarrow, "\psi"'] \arrow[d, "\pi"]\\
(A',F_{A'}) \arrow[r, "\varphi"'] & (A,F_A)
\end{tikzcd}
$$
for a morphism $\psi:(G,F)\rightarrow (A',F_{A'})$. So, we have that
\begin{align*}
J'\in F_{A'}(\{e_{A'}\})&\subseteq F(\psi^{-1}[\{e_{A'}\}])\\
&=F(\psi^{-1}[\{\varphi^{-1}(e_A)\}])\\
&=F(\pi^{-1}[\{e_A\}])\\
&=F(N)
\end{align*}
which contradicts the assumption that $J'\notin F(N)$.
\end{proof}

\begin{theorem}\label{thm:FSIP=SIP=IP+homogeneous_sorting_data}
For a sorted profinite group $(G,F)$, the following are equivalent:
\begin{enumerate}
	\item $(G,F)$ has SEP.
	\item $(G,F)$ has FSEP.
	\item $G$ has EP and for all $N,N'\in \CN(G)$ with $G/N\cong G/N'$, $F(N)=F(N')$
\end{enumerate}
\end{theorem}
\begin{proof}
It is clear that $(1)\Rightarrow (2)$ by definition and $(2)\Rightarrow (3)$ by Lemma \ref{lem:descritpion_pushforward_sortingdata_for_isoimage}. It is enough to show that $(3)\Rightarrow (1)$.

Take $(A,F_A)$ and $(B,F_B)$ in $\SIM(G,F)$, and take two morphisms $\pi_A:(G,F)\rightarrow (A,F_A)$ and $\pi:(B,F_B)\rightarrow (A,F_A)$ arbitrary. Since $G$ has EP, there is an epimorphism $\pi_B':G\rightarrow B$ such that  
$$
\begin{tikzcd}
& G \arrow[dl, dashrightarrow, "\pi_B'"'] \arrow[d, "\pi_A"]\\
B \arrow[r, "\pi"'] & A
\end{tikzcd}
$$
It is enough to show that $\pi_B'$ is sorted, that is, for each $N_B\in \CN(B)$, $$F_B(N_B)\subseteq F((\pi_B')^{-1}[N_B]).$$ Since $(B,F_B)\in \SIM(G,F)$, there is a morphism $\pi_B:(G,F)\rightarrow (B,F_B)$. Take $N_B\in \CN(B)$ arbitrary. Then, we have
$$G/(\pi_B')^{-1}[N_B]\cong B/N_B\cong G/\pi_B^{-1}[N_B],$$ and $F_B(N_B)\subseteq F(\pi_B^{-1}[N_B])$. So, by Lemma \ref{lem:descritpion_pushforward_sortingdata_for_isoimage}, 
$$F_B(N_B)\subseteq F(\pi_B^{-1}[N_B])= F((\pi_B')^{-1}[N_B]).$$

\end{proof}

\subsection{Sorted complete system}\label{Section:sorted_complete_system}
We recall the notion of {\em sorted complete system} from \cite[Section 3.2]{HoLee}. The sorted complete system of a sorted profinite group is a first order structure to encode the inverse system of finite quotients of a sorted profinite group by its open normal subgroups as the complete system of a profinite group encodes the inverse system of finite quotients of the profinite group by its open normal subgroups.

For each sorted profinite group in the category $\SPG$ , we associate a dual object, called a {\em sorted complete system}. Consider the following first order language $\CL_{SCS}(\CJ)$ with the sorts $m(k,J)$ for each $(k,J)\in \BN\times \CJ^{<\BN}$ together with
\begin{itemize}
	\item a family of binary relations $\le_{k,k',J,J'}$ and $C_{k,k',J,J'}$; and
	\item a family of ternary relations $P_{k,J}$.
\end{itemize}
For a sorted profinite group $(G,F)$, the sorted complete system $\CS(G,F)$ is a $\CL_{SCS}(\CJ)$-structure given as follows:
\begin{itemize}
	\item For $(k,J)\in \BN\times \CJ^{<\BN}$, $$m(k,J):=\bigcup_{N\in\CN(G),[G:N]\le k,J\in F(N)} G/N\times \{k\}.$$
	\item For $(k,J),(k',J')\in \BN\times \CJ^{<\BN}$,
 $$\le_{k,k',J,J'}:=\{\left((gN,k),(g'N',k')\right)\in m(k,J)\times m(k',J'):k\ge k',N\subseteq N'\}.$$
 	\item For $(k,J),(k',J')\in \BN\times \CJ^{<\BN}$,
	$$C_{k,k',J,J'}:=\{\left((gN,k),(g'N',k')\right)\in m(k,J)\times m(k',J'):k\ge k',gN\subseteq g'N'\}.$$

	\item For $(k,J)\in \BN\times \CJ^{<\BN}$,
	$$P_{k,J}=\{ \left((g_1N,k),(g_2N,k),(g_3N,k) \right)\in m(k,J)^3:g_3N=g_1g_2N\}.$$
\end{itemize}
If there is no confusion, we omit the subscripts and write $\le$ ,$C$, and $P$. We also write $gN$ for $(gN,k)$. Sorted complete systems are axiomatized by a $\CL_{SCS}(\CJ)$-theory, $SCS$ in \cite[Definition 3.7]{HoLee}.

Conversely, any model $S$ of $SCS$ is a sorted complete system of a sorted profinite group, denoted by $(G(S),F(S))$. Let $\sim$ be the equivalence relation on $S$ given as follows: For $a,b\in S$, $$a\sim b\Leftrightarrow a\le b\wedge b\le a.$$ For $a\in S$, let $[a]$ be the $\sim$-class of $a$. Then, for each $a\in m(k,J)$, $[a]\cap m(k,J)$ forms a group whose group operation is induced from $P$. The profinite group $G(S)$ is the inverse limit of the group $[a]\cap m(k,J)$ with the transition maps induced from $C$. Note that for each $N\in \CN(G(S))$, there is $a\in m(k,J)$ such that $N$ is the kernel of the projection from $G(S)$ to $[a]\cap m(k,J)$. In this case, we denote $N$ by $N_a$. We now associate the sorting data $F(S)$ on $G(S)$ as follows: For $N\in \CN(G(S))$ and $J\in \CJ^{<\BN}$, $$J\in F(S)(N)\Leftrightarrow\exists a\in m(k,J)(N=N_a).$$ Then, the sorted complete system of $(G(S),F(S))$ is naturally isomorphic to $S$. We write $\CG(S)$ for $(G(S),F(S))$.

Moreover, the associations $\mathcal S$ and $\CG$ define contravariant functors to make the category $\SPG$ of sorted profinite groups and the category of sorted complete systems whose morphisms are $\CL_{SCS}(\CJ)$-embeddings equivalent. For more detailed information, see \cite[Section 3.2]{HoLee}.

If necessary, we write $m(k,J)(S)$ for $m(k,J)$ to emphasize the sort $m(k,J)$ in a model $S$ of $SCS$.

\section{Universal SEP-cover}\label{Section:universal_SIP_cover}
Since the category $\SPG$ is closed under taking the inverse limit and the fibre product, we can transfer many arguments for profinite groups in \cite[Section 1]{HL} and in \cite[Section 24.4]{FJ} into the case of sorted profinite groups after modifying several notions properly. In this section, we aim to show that any sorted profinite group $(G,F)$ has a universal SEP-cover, generalizing \cite[Theorem 1.12]{HL} and \cite[Proposition 24.4.5]{FJ}.
\begin{definition}\label{def:SIP_cover}
We say that a morphism $p:(H,F_H)\rightarrow (G,F)$ is a {\em SEP-cover} if $(H,F_H)$ has SEP.
\end{definition}

\begin{definition}\label{def:universal_SIP_cover}
Let $(G,F)$ be a sorted profinite group. A {\em universal SEP-cover} $p:(H,F_H)\rightarrow (G,F)$ is a SEP-cover satisfying the following property: For any SEP-cover $r:(H',F_{H'})\rightarrow (G,F)$, there is a morphism $q:(H',F_{H'})\rightarrow (H,F_{H})$ such that $p\circ q=r$.
\end{definition}

\begin{remark}\label{rem:reduct_universal_cover}
Let $p:(G',F')\rightarrow (G,F)$ be a universal SEP-cover of a sorted profinite group $(G,F)$. Then, $p:G'\rightarrow G$ is the universal embedding cover of $G$ in the category of profinite groups. Namely, let $q:G''\rightarrow G$ be an embedding cover. Let $F''$ be the full sorting data on $G''$. Then, $q:(G'',F'')\rightarrow (G,F)$ is a SEP-cover. Since $p:(G',F')\rightarrow (G,F)$ is the universal SEP-cover, there is a morphism $r:(G'',F'')\rightarrow (G',F')$ such that $q=p\circ r:(G'',F'')\rightarrow (G,F)$, which implies $q=p\circ r:G''\rightarrow G$. Thus, $p:G'\rightarrow G$ is the universal embedding cover of $G$.
\end{remark}

Before showing that any sorted profinite group has a universal SEP-cover, we first introduce two posets $\CP$ and $\CH$ (see \cite[p. 188]{HL} or \cite[Section 24.4]{FJ}). Let $(G_1,F_1)$ and $(G_2,F_2)$ be sorted profinite groups. We consider the following class of pairs of morphisms with common images:
\begin{align*}
\CP&:=\CP((G_1,F_1),(G_2,F_2))\\
&=\{(\Pi_1,\Pi_2):\Pi_1:(G_1,F_1)\rightarrow (A,F_A),\Pi_2:(G_2,F_2)\rightarrow (A,F_A)\}.
\end{align*}
and define a pre-order $\le$ on $\CP$ as follows: For $(\Pi_1,\Pi_2), (\Pi_1',\Pi_2')\in \CP$, we write $(\Pi_1,\Pi_2)\le (\Pi_1',\Pi_2')$ if there is a morphism $\Pi:(A',F_{A'})\rightarrow (A,F_A)$ such that the following diagram is commutative:
$$
\begin{tikzcd}
(G_1,F_1) \arrow[dr, "\Pi_1'"] \arrow[ddr, "\Pi_1"'] &  & (G_2,F_2) \arrow[dl, "\Pi_2'"'] \arrow[ddl, "\Pi_2"]\\
& (A',F_{A'}) \arrow[d, "\Pi"]& \\
& (A,F_A)&
\end{tikzcd}
$$
We write $(\Pi_1,\Pi_2)\approx (\Pi_1',\Pi_2')$ if $(\Pi_1,\Pi_2)\le (\Pi_1',\Pi_2')$ and $(\Pi_1,\Pi_2)\ge (\Pi_1',\Pi_2')$. Then, the relation $\approx$ is an equivalence relation on $\CP$ and $\le$ gives a partial order, still denoted by $\le$, on the quotient set $\CP/\approx$.
\begin{remark}
$(\Pi_1,\Pi_2)\approx (\Pi_1',\Pi_2')$ if and only if $\Pi$ is an isomorphism.
\end{remark}
\begin{proof}
It is enough to show that the left-to-right implication holds. Suppose $(\Pi_1,\Pi_2)\approx (\Pi_1',\Pi_2')$ and $\Pi$ is not an isomorphism. First, note that $\Pi$ is bijective. Let $F:=\Pi_*(F_{A'})$ be the push-forward sorting data on $A$. Since $\Pi$ is not an isomorphism, $F_A\subset F$. Put $\Pi_1'':=\Pi\circ \Pi_1':(G_1,F_1)\rightarrow (A,F)$ and $\Pi_2'':=\Pi\circ \Pi_2'$. Since $\Pi:(A',F_{A'})\rightarrow (A,F)$ is an isomorphism, $$(\Pi_1,\Pi_2)\approx(\Pi_1',\Pi_2')\approx (\Pi_1'',\Pi_2'').$$ So, there is a morphism $\Pi':(A,F_A)\rightarrow (A,F)$ such that $\Pi'\circ \Pi_1=\Pi_1$ and $\Pi'\circ \Pi_2=\Pi_2$, which implies that $\Pi'$ is the identity map. Thus, we have that $F\subseteq F_A$, a contradiction. 
\end{proof}

We introduce a dual notion to $\CP$. Let $(G_1,F_1)\times (G_2,F_2)$ be the fibre product of $(G_1,F_1)$ and $(G_2,F_2)$ over the trivial group. Let $p_i:(G_1,F_1)\times (G_2,F_2)\rightarrow (G_i,F_i)$ for $i=1,2$ be the canonical projection. Put
\begin{align*}
\CH&:=\CH((G_1,F_1),(G_2,F_2))\\
&=\{\left((H,F_H),\Pi_1,\Pi_2\right):H\le G_1\times G_2, p_i(H)=G_i, i=1,2\}
\end{align*}
such that 
\begin{itemize}
	\item $p_i:(H_,F_H)\rightarrow (G_i,F_i)$ is a morphism for $i=1,2$;
	\item $(\Pi_1,\Pi_2)\in \CP$ with common a image $(A,F_A)$;	
	\item The following diagram is cartesian:
$$
\begin{tikzcd}
(H,F_H) \arrow[d, "p_1"'] \arrow[r, "p_2"]& (G_2,F_2) \arrow[d, "\Pi_2"]\\
(G_1,F_1) \arrow[r, "\Pi_1"']& (A,F_A)
\end{tikzcd}.
$$
\end{itemize}
By Remark \ref{rem:inducing_fibre_prod}, $\CH$ is not empty. We define a pre-order $\le'$ on $\CH$ as follows: $$\left((H,F_H),\Pi_1,\Pi_2\right)\le' \left((H',F_{H'}),\Pi_1',\Pi_2'\right)$$  if 
\begin{itemize}
	\item $(\Pi_1',\Pi_2')\le (\Pi_1,\Pi_2)$;
	\item $H\subseteq H'$.
\end{itemize}
Note that $F_H=F_{H'}$ if $H=H'$ because $H$ and $H'$ are fibre products. We write $\left((H,F_H),\Pi_1,\Pi_2\right)\approx' \left((H',F_{H'}),\Pi_1',\Pi_2'\right)$ if $H=H'$ and $(\Pi_1,\Pi_2)\approx (\Pi_1',\Pi_2')$, that is, $((H,F_H),\Pi_1,\Pi_2)\le' ((H',F_{H'}),\Pi_1',\Pi_2')$ and $((H',F_{H'}),\Pi_1',\Pi_2')\le' ((H,F_H),\Pi_1,\Pi_2)$. Then, the relation $\approx'$ is an equivalence relation on $\CH$ and $\le'$ gives a partial order on the quotient set $\CH/\approx'$.

Now we define a map $T:\CP\rightarrow \CH$ given as follows: For $(\Pi_1,\Pi_2)\in \CP$ with $(A,F_A)=\Im\Pi_1=\Im\Pi_2$, let $$T(\Pi_1,\Pi_2):=\left((G_1,F_1)\times_{(A,F_A)}(G_2,F_2), \Pi_1,\Pi_2\right).$$ The map $T$ induces a map from $\CP/\approx$ to $\CH/\approx'$. By abusing notation, we denote $\CP/\approx$, $\CH/\approx'$, and the induced map by $\CP$, $\CH$, and $T$ respectively. Note that the map $T$ is an order-reversing injection by definition. Then, we have a result analogous to \cite[Lemma 1.7]{HL} using Remark \ref{rem:inducing_fibre_prod}.
\begin{lemma}\label{lem:maximal_to_minimal}\cite[Lemma 1.7]{HL}
The map $T$ induces an order-reversing bijection between two posets $\CP/\approx$ and $\CH/\approx'$.
\end{lemma}
\begin{remark}\label{rem:extending_fibreprod_into_CH}
Let $(G_1,F_1)$ and $(G_2,F_2)$ be sorted profinite groups. Let $p_i:G_1\times G_2\rightarrow G_i$ be the canonical projection for $i=1,2$. Let $H'\le H\le G_1\times G_2$ such that $p_i[H]=G_i$ and $p_i[H']=G_i$ for $i=1,2$. For any $((H,F),\Pi_1,\Pi_2)\in \CH$, there is $(\Pi_1',\Pi_2')$ in $\CP$ such that $$((H',F'),\Pi_1',\Pi_2')\le'((H,F),\Pi_1,\Pi_2),$$ where $(H',F')$ is the fibre product of $(G_1,F_1)$ and $(G_2,F_2)$ along $\Pi_1'$ and $\Pi_2'$.
\end{remark}
\begin{proof}
Since $H'\le H$, by \cite[Lemma 1.7]{HL}, there are epimorphisms $\Pi_1'$ and $\Pi_2'$ with $\Im(\Pi_1')=\Im(\Pi_2')=:A'$ and an epimorphism $\Pi:A'\rightarrow A$ such that the following diagram is commutative:
$$
\begin{tikzcd}
H' \arrow[d, "p_1"'] \arrow[r, "p_2"] & G_2 \arrow[d, "\Pi_2'"] \arrow[ddr, bend left, "\Pi_2"]& \\
G_1 \arrow[r, "\Pi_1'"'] \arrow[drr, bend right, "\Pi_1"'] & A' \arrow[dr, "\Pi"] & \\
& & A
\end{tikzcd}
$$
where the square is cartesian. Let $F_{A'}$ be a sorting data on $A'$ such that all epimorphisms $\Pi_1'$, $\Pi_2'$, and $\Pi$ are sorted. Let $F'$ be the sorting data on $H'$ such that $(H',F')$ is the fibre product of $(G_1,F_1)$ and $(G_2,F_2)$ over $(A',F_{A'})$. Then, by the choice of $F'$ and $F_{A'}$, we have that $$((H',F'),\Pi_1',\Pi_2')\le' ((H,F),\Pi_1,\Pi_2).$$
\end{proof}

\noindent Using Zorn's Lemma with the inverse limit, we have the following result:
\begin{lemma}\label{lem:existing_maximal_in_CP}\cite[Lemma 1.8]{HL}
For every $(\Pi_1,\Pi_2)\in \CP$, there is a maximal element $(\Pi_1',\Pi_2')\in \CP$ such that $(\Pi_1,\Pi_2)\le (\Pi_1',\Pi_2')$.
\end{lemma}

We introduce a notion of the {\em quasi SEP-cover} for sorted profinite groups, analogous to the quasi-embedding cover of profinite groups in \cite[p. 189]{HL} or the $I$-cover in \cite[Definition 24.4.3]{FJ}.
\begin{definition}\label{def:quasi_embedding_cover}
A morphism $p:(H,F_H)\rightarrow (G,F)$ is called a {\em quasi SEP-cover} (q.s.c.) if for every SEP-cover $\varphi :(E,F_E)\rightarrow (G,F)$, there is a morphism $\psi:(E,F_E)\rightarrow (H,F_H)$ such that $p\circ \psi=\varphi$.
\end{definition}

\begin{remark}\label{rem:basic_property_qsc}
Let $(G,F)$ be a sorted profinite group.
\begin{enumerate}
	\item For two morphisms $p:(H,F_H)\rightarrow (G,F)$ and $\Pi:(G,F)\rightarrow (A,F_A)$ of sorted profinite groups, if both $p$ and $\Pi$ are q.s.c., then $\Pi\circ p$ is a q.s.c. 
	\item There is a cardinal $\kappa$ depending only on the cardinality of $G$ such that for any q.s.c. $p:(H,F_H)\rightarrow (G,F)$, the cardinality of $H$ is less than $\kappa$. Furthermore, if $G$ is finite, we can take $\kappa$ as finite. 
	\item Let $p:(H,F_H)\rightarrow (G,F)$ be a q.s.c. which is a SEP-cover. Then, $p$ is a universal SEP-cover.
\end{enumerate}
\end{remark}
\noindent From the proof of \cite[Lemma 24.4.4]{FJ}, we have the following result.

\begin{lemma}\label{lem:maximal_in_P_qsc}
Let $(G,F)$ be a sorted profinite group and let $(B,F_B)\in \SIM(G,F)$. Let $(\Pi_1,\Pi_2)\in \CP((B,F_B),(G,F))$ be a maximal element. Consider the following cartesian diagram induced from $(\Pi_1,\Pi_2)$:
$$
\begin{tikzcd}
(H,F_H) \arrow[d, "p_1"'] \arrow[r, "p_2"] & (G,F) \arrow[d, "\Pi_2"]\\
(B,F_B) \arrow[r, "\Pi_1"'] & (A,F_A)
\end{tikzcd}
$$
Then, $p_2$ is a q.s.c.
\end{lemma}
\begin{proof}
Let $\psi_2:(G',F')\rightarrow (G,F)$ be a SEP-cover. Then, $(B,F_B)$ is in $\SIM(G',F')$. Since $(G',F')$ has SEP, there is a morphism $\psi_1:(G',F')\rightarrow (B,F_B)$ such that $\Pi_1\circ \psi_1=\Pi_2 \circ \psi_2$. Since $H$ is a fibre product of $B$ and $G$ over $A$, by Remark \ref{rem:char_fired_prod}(2), there is a homomorphism $\psi: G'\rightarrow H$ such that the following diagram is commutative:
$$
\begin{tikzcd}
(G',F')
\arrow[dr, "\psi"]
\arrow[drr, bend left, "\psi_2"]
\arrow[ddr, bend right, "\psi_1"'] & & \\
& (H,F_H) \arrow[d, "p_1"'] \arrow[r, "p_2"] & (G,F) \arrow[d, "\Pi_2"]\\
& (B,F_B) \arrow[r, "\Pi_1"'] & (A,F_A)
\end{tikzcd}
$$
Since $\psi[G']\le H$ and $((H,F_H),\Pi_1,\Pi_2)$ is minimal, we have that $\psi[G']=H$ so that $\psi$ is an epimorphism. Indeed, if $\psi[G']\lneq H$, then by Remark \ref{rem:extending_fibreprod_into_CH}, there is $(\Pi_1',\Pi_2')$ in $\mathcal P$ such that $((\psi[G'],F_{\psi[G']}),\Pi_1',\Pi_2')\lneq' ((H,F_H),\Pi_1,\Pi_2)$, where $(\psi[G'],F_{\psi[G']})$ is the fibre product of $(B,F_B)$ and $(G,F)$ along $\Pi_1'$ and $\Pi_2'$. This contradicts  the minimality of $((H,F_H),\Pi_1,\Pi_2)$.

We will show that $\psi$ is sorted. Since $(H,F_H)$ is the fibre product of $(B,F_B)$ and $(G,F)$, the sorting data $F_H$ is generated by the following pre-sorting data $F_X$ (see Definition \ref{def:fibre_product}):
\begin{itemize}
	\item $X=\{p_1^{-1}[N_1]:N_1\in \CN(B)\}\cup\{p_2^{-1}[N_2]:N_2\in \CN(G)\}$;
	\item For $N_1\in \CN(B)$ and $N_2\in \CN(G)$, $$F_X(p_1^{-1}[N_1])=F_1(N_1),\ F_X(p_2^{-1}[N_2])=F_2(N_2).$$
\end{itemize}
Since $\psi_1$ and $\psi_2$ are sorted, for $N_1\in \CN(B)$ and $N_2\in \CN(G)$, $$F_1(N_1)\subseteq F'(\psi_1^{-1}[N_1]),\ F_2(N_2)\subseteq F'(\psi_2^{-1}[N_2]).$$ Since $\psi_1=p_1\circ \psi$ and $\psi_2=p_2\circ \psi$, we have that for $N_1\in \CN(B)$ and $N_2\in \CN(G)$, 
\begin{align*}
F'(\psi^{-1}[p_1^{-1}[N_1]])&=F'(\psi_1^{-1}[N_1])\\
&\supseteq F_1(N_1)\\
&=F_H(p_1^{-1}[N_1]),
\end{align*}
and
\begin{align*}
F'(\psi^{-1}[p_2^{-1}[N_2]])&=F'(\psi_2^{-1}[N_2])\\
&\supseteq F_2(N_2)\\
&=F_H(p_2^{-1}[N_2]),
\end{align*}
which implies $\psi$ is sorted because $F_H$ is generated by $F_X$. Therefore, we have that $\psi_2=p_2\circ \psi$ for a morphism $\psi$, and $p_2$ is a q.s.c. 
\end{proof}
\noindent We have the following characterization of sorted profinite groups not having SEP, analogous to \cite[Lemma 24.4.4]{FJ}.
\begin{lemma}\label{lem:char_no_SIP}
If a sorted profinite group $(G,F)$ does not have SEP, then, either 
	\begin{enumerate}
		\item there exists a q.s.c. $p:(H,F_H)\rightarrow (G,F)$ with a non-trivial kernel, or
		\item there is a q.s.c. $\id:(G,F')\rightarrow (G,F)$ with $F\subset F'$.
	\end{enumerate}
\end{lemma}
\begin{proof}
Suppose a sorted profinite group $(G,F)$ does not have SEP. So, there exist
\begin{itemize}
	\item $(A,F_A),(B,F_B)\in \SIM(G,F)$; and
	\item morphisms $\pi_1:(B,F_B)\rightarrow (A,F_A)$ and $\pi_2 :(G,F)\rightarrow (A,F_A)$,
\end{itemize}
such that there is no morphism $p:(G,F)\rightarrow (B,F_B)$ with $\pi_2=\pi_1\circ p$. By Lemma \ref{lem:existing_maximal_in_CP}, there is a maximal element $(\Pi_1,\Pi_2)\in \CP((B,F_B),(G,F))$ such that $(\pi_1,\pi_2)\le (\Pi_1,\Pi_2)$. Then, we have the following diagram:

$$
\begin{tikzcd}
(H,F_H) \arrow[d, "p_1"'] \arrow[r, "p_2"] & (G,F) \arrow[d, "\Pi_2"] \arrow[ddr, bend left, "\pi_2"]& \\
(B,F_B) \arrow[r, "\Pi_1"'] \arrow[drr, bend right, "\pi_1"'] & (A',F_{A'}) \arrow[dr, "\pi"] & \\
& & (A,F_A)
\end{tikzcd},
$$
where $(H,F_H)$ is the fibre product of $(B,F_B)$ and $(G,F)$ over $(A',F_{A'})$.

Note that $p_2$ is a q.s.c. by Lemma \ref{lem:maximal_in_P_qsc}. Suppose $p_2$ is an isomorphism. Let $p=p_1\circ p_2^{-1}$. Then, we have that
\begin{align*}
\pi_1\circ p&=\pi_1\circ (p_1\circ p_2^{-1})\\
&=\left ((\pi\circ \Pi_1)\circ p_1\right) \circ p_2^{-1}\\
&=\pi\circ (\Pi_2\circ p_2)\circ p_2^{-1}\\
&=\pi\circ \Pi_2\\
&=\pi_2,
\end{align*}
which is a contradiction. So, $p_2$ is not an isomorphism. If $p_2$ has a non-trivial kernel, then $p_2$ is the desired q.s.c.. Suppose $p_2$ has trivial kernel. Consider a sorting data $F'$ on $G$ given as follows: For $N\in \CN(G)$, $$F'(N):=F_H(p_2^{-1}[N]).$$ Then, $p_2 : (H,F_H)\rightarrow (G,F')$ is an isomorphism and a q.s.c.. So, by Remark \ref{rem:basic_property_qsc}(1), $\id=p_2^{-1}\circ p_2 :(G,F')\rightarrow (H,F_H)\rightarrow (G,F)$ is a q.s.c. with $F\subset F'$.

\end{proof}
Motivated by the proof of \cite[Theorem 1.12]{HL}, we provide the following lemma. 
\begin{lemma}\label{lem:inverse_limit_qsc}
For an ordinal $\alpha$, consider an inverse system $((G_i,F_i))_i$ indexed by ordinals $i< \alpha$ such that
\begin{itemize}
	\item for each $i<j$, the transition map $\pi_{j,i}$ is a q.s.c.;
	\item for each limit ordinal $\beta$, $(G_{\beta},F_{\beta})$ is the inverse limit of the inverse system $((G_i,F_i))_{i<\beta}$ with transition maps $\pi_{j,i}$ for $i<j<\beta$;
	\item for each limit ordinal $\beta$ and for $i<\beta$, the transition map $\pi_{\beta,i}$ is the natural projection from $(G_{\beta},F_{\beta})$ to $(G_i,F_i)$ coming from the inverse limit construction.
\end{itemize}
Let $(G,F)$ be the inverse limit of $((G_i,F_i))_{i<\alpha}$ and let $\pi_i:(G,F)\rightarrow (G_i,F_i)$ be the canonical projection for each $i<\alpha$. Then, $\pi_0$ is a q.s.c.
\end{lemma}
\begin{proof}
If $\alpha$ is a successor ordinal, that is, $\alpha=\alpha'+1$, then $(G,F)=(G_{\alpha'},F_{\alpha'})$ and we are done. We assume that $\alpha$ is a limit ordinal. Let $p:(G',F')\rightarrow (G_0,F_0)$ be a SEP-cover. To show that $\pi$ is a q.s.c., using transfinite induction, we will construct a sequence $(p_i:(G',F')\rightarrow (G_i,F_i))_{i<\alpha}$ of morphisms such that for each $0\le i<j<\alpha$, $p_i=\pi_{j,i}\circ p_j$. Put $p_0:=p$. Suppose we have constructed $(p_i)_{i< \gamma}$ for some $\gamma<\alpha$. If $\gamma$ is a limit ordinal, there is a desired morphism $p_{\gamma}:(G',F')\rightarrow (G_{\gamma},F_{\gamma})$ because $(G_{\gamma},F_{\gamma})$ is the inverse limit of $(G_i,F_i)_{i<\gamma}$. If $\gamma=\gamma'+1$ is a successor ordinal, there is a morphism $r:(G',F')\rightarrow (G_{\gamma},F_{\gamma})$ such that $p_{\gamma'}=\pi_{\gamma,\gamma'}\circ r$ because $\pi_{\gamma,\gamma'}$ is a q.s.c. Put $p_{\gamma}:=r$, which is the desired one. 

Since $(G,F)$ is the inverse limit of $(G_i,F_i)_{i<\beta}$, there is $q:(G',F')\rightarrow (G,F)$ such that for each $i<\alpha$, $p_i=\pi_i\circ q$. Thus, we have that $p=\pi_0\circ q$, and $\pi_0$ is a q.s.c.
\end{proof}

\begin{theorem}\label{thm:existence_universal_SIP_cover}
Let $(G,F)$ be a sorted profinite group. Then, there is a universal SEP-cover $p:(H,F_H)\rightarrow (G,F)$. Furthermore, if $G$ is finitely generated, then $p$ is the unique universal SEP-cover (up to isomorphism).
\end{theorem}
\begin{proof}
If $(G,F)$ has SEP, then $\id : (G,F)\rightarrow (G,F)$ is a universal SEP-cover. Suppose $(G,F)$ does not have SEP. Let $\kappa$ be a cardinal such that any q.s.c. $(H,F_H)$ of $(G,F)$ has the cardinality less than $\kappa$ and let $\aleph$ be a cardinal with $2^{\kappa}<\aleph$.

Using transfinite induction, we will construct an inverse system $((G_i,F_i))_i$ indexed by ordinals $i\le \alpha$ for some ordinal $\alpha<\aleph$ such that
\begin{enumerate}
	\item for each $i<j$, the transition map $\pi_{j,i}$ is a q.s.c. with a non-trivial kernel;
	\item for each limit ordinal $\beta$, $(G_{\beta},F_{\beta})$ is the inverse limit of the inverse system $((G_i,F_i))_{i<\beta}$ with transition maps $\pi_{i,j}$ for $i<j<\beta$;
	\item for each limit ordinal $\beta$ and for $i<\beta$, the transition map $\pi_{\beta,i}$ is the natural projection from $(G_{\beta},F_{\beta})$ to $(G_i,F_i)$ coming from the inverse limit construction;
	\item any q.s.c. to $(G_{\alpha},F_{\alpha})$ has trivial kernel.
\end{enumerate}
Put $(G_0,F_0)=(G,F)$. Suppose we have constructed an inverse system $((G_i,F_i))_{i<\beta}$ for an ordinal $\beta$, which satisfies $(1)-(3)$.\\

Case. $\beta$ is a successor ordinal, that is, $\beta=\beta'+1$. If $(G_{\beta'},F_{\beta'})$ has SEP, then we stop the process. Suppose $(G_{\beta'},F_{\beta'})$ does not have SEP. By Lemma \ref{lem:char_no_SIP}, there is a q.s.c. to $(G_{\beta'},F_{\beta'})$, which is not an isomorphism. If every q.s.c. to $(G_{\beta'},F_{\beta'})$ is injective, we stop the process. So, we may assume that there is a q.s.c. $p:(G',F')\rightarrow (G_{\beta'},F_{\beta'})$ with a non-trivial kernel. Put $(G_{\beta},F_{\beta}):=(G',F')$ and $\pi_{\beta,\beta'}:=p$. For each $i<\beta$, put $\pi_{\beta,i}:=\pi_{\beta',i}\circ p$. By Remark \ref{rem:basic_property_qsc}(1), each $\pi_{\beta,i}$ is a q.s.c.\\

Case. $\beta$ is a limit ordinal. Let $(G_{\beta},F_{\beta})$ be the inverse limit of $((G_i,F_i))_{i<\beta}$. For each $i<\beta$, let $\pi_{\beta,i}$ be the natural projection map from $(G_{\beta},F_{\beta})$ to $(G_i,F_i)$. By Lemma \ref{lem:inverse_limit_qsc}, each $\pi_{\beta,i}$ is a q.s.c.\\

For each $i<j$, $\pi_{j,i}$ has a non-trivial kernel. Namely, suppose there are $i<j$ such that $\ker(\pi_{j,i})$ is trivial. Since $\pi_{j,i}=\pi_{i+1,i}\circ \pi_{j,i+1}$, where $\pi_{k,k}=\id$ for each $k$, $\ker(\pi_{i+1,i})$ is also trivial, which is a contradiction. In our construction, $\beta$ should be less than $\aleph$. Suppose $\beta\ge \aleph$. Without loss of generality, we may assume that $\beta$ is a successor ordinal and write $\beta=\beta'+1$. We have that $|G_{\beta'}|\le \kappa$. So, $|\CN(G_{\beta'})|\le 2^{\kappa}$. Since $|\beta'|\ge \aleph>2^{\kappa}$, by the pigeon hole principle, for some $i<j<\beta'$, $\ker(\pi_{\beta',j})=\ker(\pi_{\beta',i})$. Since $\pi_{\beta',i}=\pi^j_i\circ \pi_{\beta',j}$, we have that $\pi_{j,i}$ is injective, which is a contradiction. 

Now, take a maximal inverse system $((G_i,F_i))_{i<\beta}$ satisfying $(1)-(3)$. Note that $\beta$ is a successor ordinal and write $\beta=\alpha+1$. Then, for the q.s.c. $\pi_{\alpha,0}:(G_{\alpha},F_{\alpha})\rightarrow (G,F)$, we have that every q.s.c. to $(G_{\alpha},F_{\alpha})$ is injective and so $((G_i,F_i))_{i\le \alpha}$ is the desired one.\\

Let $(E(G),E(F)):=(G_{\alpha},F_{\alpha})$ and let $\pi:=\pi_{\alpha,0}$. If $(E(G),E(F))$ has SEP, then $\pi$ is a universal SEP-cover of $(G,F)$. Suppose $(E(G),E(F))$ does not have SEP. By transfinite induction, we will construct a sequence $(E^i(F))_{i\le \gamma'}$ of sorting data on $E(G)$ such that for $i<j\le \gamma'$
\begin{itemize}
	\item $E^i(F)\subset E^j(G)$;
	\item $\id_{j,i}:=\id:(E(G),E^j(F))\rightarrow (E(G),E^i(F))$ is a q.s.c.;
	\item $(E(G),E^{\gamma'}(F))$ has SEP.
\end{itemize}
Suppose we have constructed such a sequence $(E^i(F))_{i< \alpha'}$ for an ordinal $\alpha'$. If $\alpha'$ is a limit ordinal, put $E^{\alpha'}(F):=\bigcup_{i<\alpha'}E^i(F)$. By Lemma \ref{lem:inverse_limit_qsc}, each $\id_{\alpha',i}:(E(G),E^{\alpha'}(F))\rightarrow (E(G),E^i(F))$ is a q.s.c.

Suppose $\alpha'$ is a successor ordinal, that is, $\alpha'=\alpha''+1$. If $(E(G),E^{\alpha''}(F))$ has SEP, put $\gamma'=\alpha''$ and stop the process. If $(E(G),E^{\alpha''}(F))$ does not have SEP, then by Lemma \ref{lem:char_no_SIP}(2), there is a sorting data $F'$ on $E(G)$ such that $E^{\alpha''}(F)\subset F'$ and $\id:(E(G),F')\rightarrow (E(G),E^{\alpha''}(F))$ is a q.s.c, and put $E^{\alpha'}(F):=F'$. Note that we can not apply Lemma \ref{lem:char_no_SIP}(1) in this case. Namely, if there is a q.s.c $\pi:(G',F')\rightarrow (E(G),E^{\alpha''}(F))$ whose kernel is non-trivial, then we have a q.s.c. $\id \circ \pi:(G',F')\rightarrow (E(G),E^{\alpha''}(F))\rightarrow (G,F)$ whose kernel is non-trivial, which contradicts the choice of $(G,F)$.

This transfinite inductive process should stop for some $\gamma'<|\CJ|^+$ so that $(E(G),E^{\gamma'}(F))$ has SEP. Thus, a q.s.c. $$\pi_{\alpha,0}\circ \id_{\gamma',0}:(E(G),E^{\gamma'}(F))\rightarrow (E(G),E(F))\rightarrow(G,F)$$ is a SEP-cover so that the SEP-cover $\pi_{\alpha,0}\circ \id_{\gamma',0}$ is a universal SEP-cover.\\

We now prove the furthermore part. Suppose $G$ is finitely generated. Let $p_i:(G_i,F_i)\rightarrow (G,F)$ be a universal SEP-cover of $(G,F)$ for $i=1,2$. By Remark \ref{rem:reduct_universal_cover} and \cite[Theorem 1.12]{HL}, $G_1$ and $G_2$ are finitely generated. By universality, there are morphisms $q:(G_2,F_2)\rightarrow (G_1,F_1)$ and $q':(G_1,F_1)\rightarrow (G_2,F_2)$ such that $p_2=p_1\circ q$ and $p_1=p_2\circ q'$. Since $G_1$ and $G_2$ are finitely generated, by \cite[Proposition 7.6]{R}, both $q$ and $q'$ are bijective. We want to show that $q_*(F_2)=F_1$. Since $q:(G_2,F_2)\rightarrow (G_1,F_1)$ is a morphism, $F_1\subseteq q_*(F_2)$. Also we have that $q:(G_2,F_2)\cong (G_1,q_*(F_2))$ and $p_2\circ q^{-1}:(G_1,q_*(F_2))\rightarrow (G,F)$ is a universal SEP-cover. So, there is a morphism $r:(G_1,F_1)\rightarrow (G_1,q_*(F_2))$ such that $p_1=(p_2\circ q^{-1})\circ r$. Since $p_2=p_1\circ q$, the morphism $r$ should be the identity map on $G$ and $q_*(F_2)\subseteq F_1$. Thus, $F_1=q_*(F_2)$.
\end{proof}

\begin{remark}\label{rem:full_subcategory_closed_under_a_universal_ec}
Let $\mathcal C$ be a formation of finite groups, that is, a set of finite groups closed under taking quotients and fibre products (see \cite[Section 17.3]{FJ}). Let $\Pro\mathcal C$ be the set of {\em pro-$\mathcal C$ groups}, that is, the set of inverse limits of groups in $\mathcal C$. Let $\Pro\mathcal C_{\CJ}$ be the full subcategory of the category $\SPG$ whose objects are of the form $(G,F)$, called a {\em sorted pro-$\mathcal C$ group}, for a pro-$\CC$ group $G$. In \cite[Lemma 24.4.6]{FJ}, any pro-$\mathcal C$ group $G$ has the universal embedding cover, which is also a pro-$\mathcal C$ group. So, by Remark \ref{rem:reduct_universal_cover}, any sorted pro-$\mathcal C$ group has a universal SEP-cover, which is also a pro-$\mathcal C$ group.

Suppose $\mathcal C$ is closed under taking subgroups. For example, let $\mathcal C$ be the set of abelian groups, nilpotent groups, solvable groups, or $p$-groups for a fixed prime $p$. By \cite[Lemma 17.3.1]{FJ}, $\Pro\mathcal C$ is closed under taking quotients, inverse limits, and fibre products. Since our proof works for any full subcategory of $\SPG$ closed under taking quotients, inverse limits, and fibre products, we also deduce from the proof of Theorem \ref{thm:existence_universal_SIP_cover} that any sorted pro-$\mathcal C$ group has an universal SEP-cover, which is also a sorted pro-$\mathcal C$ group.
\end{remark}

\begin{example}
We continue working with the notations in Example \ref{ex:sorted_profinite_group_having_SEP}(2). Define a sorting data $F$ on $G=\BZ_2\times \BZ_2$ as follows:
\begin{itemize}
	\item $F(0)=F(G)=F(N_{(1,1)})=\CJ^{<\BN}$;
	\item $F(N_{(1,0)})=F(N_{(0,1)})=\{J\in \CJ^{<\BN}:s_1\in ||J||\}$.
\end{itemize}
Take $(\Pi_1,\Pi_2)\in \CP((G,F),(G,F))$ such that
\begin{itemize}
	\item $A:=G$, $F_A$ is a sorting data on $A$ given as follows:
	\begin{itemize}
		\item $F_A(A)=\CJ^{\BN}$;
		\item $F_A(0)=F(N_{(1,0)})=F(N_{(0,1)})=F(N_{(1,1)})=\{J\in \CJ^{<\BN}:s_1\in ||J||\}$.
	\end{itemize}
	\item $\Pi_1,\Pi_2:(G,F)\rightarrow(A,F_A)$ given as follows:
	\begin{itemize}
		\item $\Pi_2=\id$;
		\item $\Pi_1:(1,1)\mapsto (1,0),(1,0)\mapsto (1,1), (0,1)\mapsto(0,1)$.
	\end{itemize}
\end{itemize}
Then, $(\Pi_1,\Pi_2)$ is a non-trivial maximal element. The fibre product with respect to $(\Pi_1,\Pi_2)$ is $(G,E^0(F))$ where the sorting data $E^0(F)$ is given as follows:
\begin{itemize}
	\item $E^0(F)(0)=E^0(F)(G)=E^0(F)(N_{(1,0)})=E^0(F)(N_{(1,1)})=\CJ^{<\BN}$;
	\item $E^0(F)(N_{(0,1)})=\{J\in \CJ^{<\BN}:s_1\in ||J||\}$.
\end{itemize}
Take $(\Pi_1',\Pi_2')\in \CP((G,E^0(F)),(G,E^0(F)))$ such that
\begin{itemize}
	\item $A':=G$, $F_{A'}$ is a sorting data on $A'$ given as follows:
	\begin{itemize}
		\item $F_{A'}(0)=F_{A'}(A')=F_{A'}(N_{(1,1)})=\CJ^{<\BN}$;
		\item $F_{A'}(N_{(1,0)})=F_{A'}(N_{(0,1)})=\{J\in \CJ^{<\BN}:s_1\in ||J||\}$.
	\end{itemize}
	\item $\Pi_1',\Pi_2':(G,E^0(F))\rightarrow (A',F_{A'})$ given as follows:
	\begin{itemize}
		\item $\Pi_2'=\id$;
		\item $\Pi_1':(1,1)\mapsto (0,1), (1,0)\mapsto (1,1), (0,1)\mapsto (1,0)$.
	\end{itemize}
\end{itemize}
Then, $(\Pi_1',\Pi_2')$ is a non-trivial maximal element. The fibre product with respect to $(\Pi_1',\Pi_2')$ is $(G,E^1(F))$ where the sorting data $E^1(F)$ is full. Thus, the identity map from $(G,E^1(F))$ to $(G,F)$ is a universal SEP-cover of $(G,F)$.
\end{example}

Using the characterization of SEP in Theorem \ref{thm:FSIP=SIP=IP+homogeneous_sorting_data}, we show that there is a unique (up to isomorphism) universal SEP-cover of a given sorted profinite group.

\begin{fact}\label{fact:unique_universal_EP-cover}\cite[Theorem 2.7]{C1}
A profinite group $H$ has the unique universal embedding cover $E(H)$ and $$\IM(E(H))=\bigcup_{A\in \IM(H)}\IM(E(A))$$ where $E(A)$ is the unique universal embedding cover of $A$.
\end{fact}

\begin{theorem}\label{thm:uniqueness_SEP-cover}
There is a unique (up to isomorphism) universal SEP-cover $(G,F)$ of a sorted profinite group $(H,F_H)$.
\end{theorem}
\begin{proof}
By Theorem \ref{thm:existence_universal_SIP_cover}, each $(A,F_A)\in \FSIM(H,F_H)$ has the unique universal SEP-cover $(E(A),E(F_A))$. Let $$\Gamma:=\bigcup_{(A,F)\in \FSIM(H,F_H)}\FSIM(E(A),E(F_A)).$$

\bigskip

Let $G$ be the unique universal embedding cover of $H$. Note that $$\IM(G)=\{A:(A,F_A)\in \Gamma \mbox{ for some sorting data }F_A\}$$ by Fact \ref{fact:unique_universal_EP-cover}. Define a pre-sorting data $\hat F$ on $G$ as follows: For each $N\in \CN(G)$, $$\hat F(N):=\bigcup_{(A,F_A)\in \Gamma, A\cong G/N}F_A(\{e_A\}).$$

\begin{claim}\label{claim:presort_is_sort}
The pre-sorting data $\hat F$ is a sorting data.
\end{claim}
\begin{proof}
It is enough to show that $\hat F$ satisfies $(3)$ and $(4)$ in Definition \ref{def:sorting_data}, that is,
\begin{enumerate}
	\item[(3)] Suppose $N_1\subseteq N_2$ and $[G:N_1]\le k$. For $J\in \CJ^{<\mathbb N}$, $$J\in \hat F(N_1)\Rightarrow J^*_{\subseteq}(k,J)\in \hat F(N_2).$$
	\item[(4)] For $J_1\in \hat F(N_1)$ and $J_2\in \hat F(N_2)$, $J^*_{\cap}(J_1,J_2)\in \hat F(N_1\cap N_2)$.
\end{enumerate}

\bigskip

$(3)$ Let $N_1\subseteq N_2\in \CN(G)$ with $[G:N_1]\le k$ and let $J\in \CJ^{<\mathbb N}$ with $J\in \hat F(N_1)$. By the choice of $\hat F$, there are 
\begin{itemize}
	\item the universal SEP-cover $(E(A),E(F_A))$ of some $(A,F_A)\in \FSIM(H,F_H)$; and
	\item $(G/N_1,F_1)\in \FSIM(E(A),E(F_A))$ such that $J\in F_1(\{e_{G/N_1}\})$.
\end{itemize}
Then, $J^*_{\subseteq}(k,J)\in F_1(N_2/N_1)$. Consider a push-forward sorting data $F_2:=\pi_*(F_1)$ on $G/N_2$ where $\pi:G/N_1\rightarrow G/N_2$ is the projection. Then, $(G/N_2,F_2)\in \FSIM(E(A),E(F_A))$ and $J^*_{\subseteq}(k,J)\in F_2(\{e_{G/N_2}\})=F_1(N_2/N_1)\subseteq \hat F(N_2)$.\\

$(4)$ Take $N_1,N_2\in \CN(G)$, $J_1\in \hat F(N_1)$, and $J_2\in \hat F(N_2)$ arbitrary. By the choice of $\hat F$, for each $i=1,2$, there are
\begin{itemize}
	\item $\bar N_i\in \CN(H)$;
	\item the universal SEP-cover $(E(G/\bar N_i),E(\bar F_i))$ of some $(G/\bar N_i,\bar F_i)\in \FSIM(H,F_H)$; and
	\item $(G/N_i,F_i)\in \FSIM(E(G/\bar N_i),E(\bar F_i))$ such that $J_i\in F_i(\{e_{G/N_i}\})$.
\end{itemize}

Let $\bar F_0$ be a sorting data on $H/(\bar N_1\bar N_2)$ such that the projection from $H/\bar N_i$ to $H/(\bar N_1\bar N_2))$ is sorted for $i=1,2$. Let $(H/(\bar N_1\cap \bar N_2),\bar F_3)$ be the fibre product of $(H/\bar N_1,\bar F_1)$ and $(H/\bar N_2,\bar F_2)$ over $(H/(\bar N_1\bar N_2), \bar F_0)$. Let $\left(E(H/(\bar N_1\cap \bar N_2)),E(\bar F_3)\right)$ be the unique universal SEP-cover of $(H/(\bar N_1\cap \bar N_2),\bar F_3)$. 

Let $F_0$ be a sorting data on $G/(N_1N_2)$ such that the projection from $G/N_i$ to $G/(N_1N_2)$ is sorted for $i=1,2$. Let $(G/(N_1\cap N_2),F_3)$ be the fibre product of $(G/N_1,F_1)$ and $(G/N_2,F_2)$ over $(G/(N_1N_2),F_0)$.

Then, there is a morphism from $(E(H/(\bar N_1\cap \bar N_2)),E(\bar F_3))$ to $(E(H/\bar N_i),E(\bar F_i))$ and so there is a morphism from $(E(H/(\bar N_1\cap \bar N_2)),E(\bar F_3))$ to $(G/N_i,F_i)$ for $i=1,2$. Since $(G/(N_1\cap N_2),F_3)$ is the fibre product of $(G/N_1,F_1)$ and $(G/N_2,F_2)$ over $(G/(N_1N_2),F_0)$, there is a morphism from $(E(H/(\bar N_1\cap \bar N_2)),E(\bar F_3))$ to $(G/(N_1\cap N_2),F_3)$. So, we have the following diagram:

$$
\begin{tikzcd}
(G/N_1,F_1)  & (E(H/\bar N_1),E(\bar F_1)) \arrow[l] \arrow[r] & (H/\bar N_1)  \\
(G/(N_1\cap N_2),F_3) \arrow[u] \arrow[d] &(E(H/(\bar N_1\cap \bar N_2)),E(\bar F_3)) \arrow[d,  dashrightarrow] \arrow[u,  dashrightarrow] \arrow[r] \arrow[l,  dashrightarrow]& (H/(\bar N_1\cap \bar N_2),\bar F_3) \arrow[u] \arrow[d]\\
(G/N_2,F_2)  & (E(H/\bar N_2),E(\bar F_2)) \arrow[l] \arrow[r] &(H/\bar N_2,\bar F_2)
\end{tikzcd}
$$

\noindent Therefore, $$(G/(N_1\cap N_2),F_3)\in \FSIM(E(H/(\bar N_1\cap \bar N_2)),E(\bar F_3))\subseteq \Gamma,$$ and so $$J^*_{\cap}(J_1,J_2)\in F_3(\{e_{G/(N_1\cap N_2)}\})\subseteq \hat F(N_1\cap N_2).$$
\end{proof}

By Theorem \ref{thm:FSIP=SIP=IP+homogeneous_sorting_data}(3), $(G,\hat F)$ has SEP. Now we show that $(G,\hat F)$ is the unique universal SEP-cover of $(H,F_H)$.

\begin{claim}\label{claim:minimality_hatF}
For any SEP-cover $(G',F')$ of $(H,F_H)$, any epimorphism $\varphi:G'\rightarrow G$ is sorted, that is, for each $N\in \CN(G)$, $$\hat F(N)\subseteq F'(\varphi^{-1}[N]).$$
\end{claim}
\begin{proof}
Let $(G',F')$ be a SEP-cover of $(H,F_H)$ and let $\varphi$ be an epimorphism from $G'$ to $G$. Take $N\in \CN(G)$ and $J\in \hat F(N)$ arbitrary. We show that $J\in F'(\varphi^{-1}[N])$. There is a sorting data $F$ on $G/N$ with $(G/N,F)\in \Gamma$ such that $J\in F(\{e_{G/N}\})$. Since $(G/N,F)\in \Gamma\subseteq \FSIM(G',F')$ and $G'/\varphi^{-1}[N]\cong G/N$, by Lemma \ref{lem:descritpion_pushforward_sortingdata_for_isoimage}, $J\in F(\{e_{G/N}\})\subseteq F'(\varphi^{-1}[N])$.
\end{proof}

Let $(G',F')$ be a SEP-cover of $(H,F_H)$. Since $G'$ is an embedding cover of $H$ and $G$ is the universal embedding cover of $H$, there is an epimorphism from $\varphi:G'\rightarrow G$. By Claim \ref{claim:minimality_hatF}, the epimorphism $\varphi$ is sorted. This implies that $(G,F)$ is a universal SEP-cover of $(H,F_H)$.

It remains to show that $(G, \hat F)$ is unique up to isomorphism. Let $(G',F')$ be a universal SEP-cover of $(H,F_H)$. Since $G'$ is a universal embedding cover of $H$, there is an isomorphism $\varphi:G'\rightarrow G$. We show that the isomorphism $\varphi:G'\rightarrow G$ induces an isomorphism $\varphi:(G',F')\rightarrow (G,\hat F)$. It is enough to show that for any $N\in \CN(G)$, $$\hat F(N)=F'(\varphi^{-1}[N]).$$ Take $N\in \CN(G)$ arbitrary. By Claim \ref{claim:minimality_hatF}, $\varphi$ is sorted so that $\hat F(N)\subseteq F'(\varphi^{-1}[N])$. Since $(G',F')$ is a universal SEP-cover of $(H,F_H)$, there is a morphism $\psi:(G,\hat F)\rightarrow (G',F')$. So, we have that 
\begin{itemize}
	\item $G/N\cong G'/\varphi^{-1}[N]\cong G/\psi^{-1}[\varphi^{-1}[N]]$; and
	\item $\hat F(N)\subseteq F'(\varphi^{-1}[N])\subseteq \hat F(\psi^{-1}[\varphi^{-1}[N]])$.
\end{itemize}
Since $G/N\cong G/\psi^{-1}[\varphi^{-1}[N]]$, by Lemma \ref{lem:descritpion_pushforward_sortingdata_for_isoimage}, $$\hat F(\psi^{-1}[\varphi^{-1}[N]])=\hat F(N)$$ and so $\hat F(N)= F'(\varphi^{-1}[N])= \hat F(\psi^{-1}[\varphi^{-1}[N]])$.

\end{proof}

\section{Model theory of sorted profinite groups having SEP}\label{Section:model theory of sorted profinite group with SIP}
In \cite[Theorem 2.4]{C1}, Chatzidakis showed that the theory of a complete system of a profinite group having embedding property is $\omega$-stable and gave a description of forking independence. In this section, our main goal is to generalize this phenomenon into the sorted profinite group having SEP when $\CJ$ is countable.

\begin{definition}\label{def:basic_terminologies_SCS}
\begin{enumerate}
	\item A {\em subsystem} of $S$ is a substructure of $S$ which is a model of $SCS$.

	\item Let $X\subseteq S$. There is the smallest subsystem $S_X$ containing $X$. In this case, we say that $S_X$ is {\em generated by $X$}.
	
	\item We say that a subsystem $S'$ is {\em finitely generated} if there is a  finite subset $X'$ such that $S'=S_{X'}$.
\end{enumerate}
\end{definition}

\begin{definition}\label{def:basic_terminoligies_2}
\begin{enumerate}	
	\item Let $X$ be a subset of $S$. We say that $X$ is {\em locally full} if for each $x\in X$ and for each $(k,J)\in \BN\times \CJ^{<\BN}$, $$[x]\cap m(k,J)\subseteq X.$$
	
	\item A subset $X$ of $S$ is called {\em relatively dense} if for each $s\in S_X$, there is $x\in X$ such that $x\le s$. 
	
	\item A subset $X$ of $S$ is called a {\em presystem} if it is locally full and relatively dense. 
\end{enumerate}
\end{definition}
\noindent Note that any subsystem is already locally full because for a complete system $S$ and for $s\in S$ and $(k,J)\in \BN\times \CJ^{<\BN}$, $$s\in m(k,J)\Leftrightarrow 1\le |[s]\cap m(k,J)|\le k.$$

\begin{remark}\label{rem:smallest subsystem}
\begin{enumerate}
	\item Let $S'$ be a subsystem generated by $X$. Then, we can take $X$ as a presystem. Also, if $X$ is finite, we can take $X$ as a finite presystem.
	
	\item If $X$ is a presystem, then $S_X\subseteq \dcl(X)$. More precisely, for each $x\in S_X\cap m(k,J)$, there is $a_x\in X$ such that $x$ is the unique element in $S$ satisfying $S_X\models C(a_x,x)\wedge x\in m(k,J)$. So, any embedding from $S_X$ to $S$ is uniquely determined by the image of $X$.

	\item Let $S'$ be a finitely generated sorted complete system. By (1), there is a finite presystem $X$ generating $X$. Without loss of generality, we may assume that $m(k,J)\cap S'\subseteq X$ for $(k,J)\in \mathbb N\times \mathcal J^{<\mathbb N}$ with $X\cap m(k,J)\neq\emptyset$. We call such a presystem {\em nice}. 
	
	\item Let $S'$ be a finitely generated sorted complete system and let $X$ be a nice finite presystem generating $S'$. Note that for any $\CL_{SCS}(\CJ)$-embedding $\Phi:S'\rightarrow S'$, the restriction $\Phi|_X$ is an automorphism of $X$. Conversely, any automorphism of $X$ determines a unique automorphism of $S'$ by (2). Also, there is a formula $\theta_X(\bar x)$ such that $SCS\models \theta_X(\bar x)\rightarrow \qftp(X)$.	

\end{enumerate}
\end{remark}

Recall that $S/\sim$ forms a distributive lattice, denoted by $(S/\sim,\vee,\wedge)$ induced from the partial order $\le$ on $S$ (cf. \cite[Section 3.2]{HoLee}).
\begin{definition}\label{def:meet_join}
Let $A$ and $B$ be subsets of $S$.
\begin{enumerate}
	\item $\min A:=\{a\in A:\forall b\in A(a\le b) \}$.
	\item $A\vee B:=\min \{c\in S:[c]=[a]\vee[b],\ a\in  A,\ b\in B\}$.
	\item We write $A\le B$ if $a\le b$ for every $a\in A$ and $b\in B$.
	\item We write $A\sim B$ if $a\sim b$ for every $a\in A$ and $b\in B$.
\end{enumerate}
\end{definition}
\noindent For $a\in S$, we write $a\vee A$ for $\{a\}\vee A$. Note that if $A\vee B\neq \emptyset$, then $A\vee B=[a]$ for any $a\in A\vee B$.

\begin{remark}\label{rem:meet_join_subsystems}
\begin{enumerate}
	\item Let $S_1$ and $S_2$ be subsystems of $S$, and let $S_0=S_1\cap S_2$. Suppose $S_1$ is finitely generated. Then, $\min S_1\neq \emptyset$, $S_1\vee S_2\neq\emptyset$, and $\min S_0\sim S_1\vee S_2$.

	\item Let $S'$ be a sorted complete system with $\min S'\neq \emptyset$. Choose $a\in \min S'\cap m(k,J)\neq \emptyset$. Then, $G(S')\cong [a]\cap m(k,J)$ and for each $x\in m(k',J')$, there is a unique $a_x\in [a]\cap m(k,J)$ such that $x$ is the unique element in $S'$ satisfying $S'\models C(a_x,x)\wedge x\in m(k',J')$.
	
	\item A sorted complete system $S'$ is finitely generated if and only if $\mathcal G(S')$ is a finitely sorted finite group.

\end{enumerate}
\end{remark}
\begin{proof}
$(1)$ Suppose $S_1$ is finitely generated by a finite subset $X$ of $S_1$. We may assume that $X$ is relatively dense. Since $X$ is finite, $\min X\neq \emptyset$, and since $X$ is relatively dense, $\min X\subset \min S_1$.

Next, we show that $S_1\vee S_2$ is not empty. Choose $a_0\in \min S_1$ arbitrary. Then, there are only finitely many elements in the set $\{x\in S:a_0\le x\}$ up to $\sim$ and so there is a minimal element $c\in S$ such that $[c]=[a_0]\vee [b]$ for some $b\in S_2$. So, $c\in S_1\vee S_2$ and $S_1\vee S_2\neq\emptyset$.

For all $a_0\in \min S_1$ and for all $c\in S_1\vee S_2$, $a_0\le c$, and so there are only finitely many elements in $S_1\vee S_2$ up to $\sim$. Thus, $S_1\vee S_2\neq \emptyset$.

It remains to show that $\min S_0\sim S_1\vee S_2$. Since $S_0\subset S_1$, we have that $\min S_0\neq \emptyset$. Take $c\in S_1\vee S_2$ and $d\in \min S_0$ arbitrary. Since $d\in S_1\cap S_2$, it is clear that $c\le d$. Take $a\in S_1\cap m(k,J)$ and $b\in S_2\cap m(k',J')$ such that $[c]=[a]\vee[b]$ and $c\in m(kk',J_{\subset}^*(k,J))$. Since $S_1$ and $S_2$ are subsystems, there is $c'\in S_1\cap S_2\cap m(kk',J_{\subset}^*(k,J)^{\frown} JJ')$ such that $c'\sim c$. Thus, $d\le c'\le c$.\\

$(2)$ Fix $a\in \min S'\cap m(k,J)$ for some $(k,J)\in \mathbb N\times \CJ^{<\mathbb N}$. Since $G(S)$ is the inverse limit of the group $[x]\cap m(k',J')$ of $x\in m(k',J')$ with the transition maps induced from $C$, $G(S)\cong [a]\cap m(k,J)$. Also for each $b\in S$, after putting $C(b,S):=\{x\in S:S\models C(b,x)\}$, $C(b,S)\cap m(k',J')$ is linearly ordered by $\le$ for each $(k',J')\in \mathbb N\times \CJ^{<\mathbb N}$. Thus, for each $x\in m(k',J')$, there is a unique $a_x\in [a]\cap m(k,J)$ such that $x$ is the unique element in $S$ satisfying $S\models C(a_x,x)\wedge x\in m(k',J')$.\\

$(3)$ Let $S'\models SCS$. Suppose $S'$ is finitely generated by a finite presystem $X$ of $S'$. Since $X$ is relatively dense in $S'$ and $X$ is finite, $\min X\neq \emptyset$ and $\min X\sim \min S'$. By Remark \ref{rem:meet_join_subsystems}(2), $G(S')$ is finite. Also, since $S'=S_{X'}$ and $X'$ is finite, $F(S')$ is finitely generated. So, $\CG(S')=(G(S'),F(S'))$ is finitely sorted. 

Conversely, suppose $\mathcal G(S')=(G(S'),F(S'))$ is a finitely sorted finite group. Since $\CN(G(S'))$ is finite, we may assume that $F(S')$ is generated by a pre-sorting data $\hat F$ such that for each $N\in \CN(G(S'))$, $\hat F(N)$ is finite. Define a presystem $X$ as follows: For $(k,J)\in \mathbb N\times \CJ^{<\mathbb N}$, $$X\cap m(k,J)=\bigcup_{N\in \CN(G(s')),[G(S'):N]\le k, J\in \hat F(N)}G(S')/N\times \{k\},$$ which is finite because $\hat F(N)$ is finite for each $N\in \CN(G(S'))$. Then, the finite presystem $X$ generates $S'$.

\end{proof}

From the fibre product in Definition \ref{def:fibre_product}, we have the following lemma.
\begin{lemma}\label{lem:fibre_prod_in_sortedcompletesystem}
Let $S_0$, $S_1$, and $S_2$ be subsystems of $S$ and let $S_3=S_{S_1\cup S_2}$. Suppose  that $\min S_1\neq\emptyset$ and $S_0\subset S_1\cap S_2$. Consider the following inclusions:
\begin{itemize}
	\item $\iota_{S_0,S_1}:S_0\rightarrow S_1$, $\iota_{S_0,S_2}:S_0\rightarrow S_2$;
	\item $\iota_{S_1,S_3}:S_{1}\rightarrow S_{3}$;
	\item $\iota_{S_2,S_3}:S_{2}\rightarrow S_{3}$.
\end{itemize}
\begin{enumerate}
	\item Suppose $S_1\vee S_2\sim \min S_0$. Then, $S_3$ is the {\em co-fibre product} of $S_1$ and $S_2$ over $S_0$, that is, we have the following cartesian diagram:
$$
\begin{tikzcd}
\CG(S_3)\arrow[rr, "\CG(\iota_{S_2,S_3})"] \arrow[d, "\CG(\iota_{S_1,S_3})"']& & \CG(S_2) \arrow[d, "\CG(\iota_{S_0,S_2})"]\\
\CG(S_1) \arrow[rr, "\CG(\iota_{S_0,S_1})"'] & & \CG(S_0)
\end{tikzcd}
$$

	\item Let $S_1'$ be a subsystem of $S$ such that $S_0\subset S_1'$ and $S_1'\vee S_2\sim \min S_0$. Suppose there is an isomorphism $f:S_{1}\rightarrow S_{1}'$ making the following diagram commute:
$$
\begin{tikzcd}
S_0 \arrow[rr, "\iota_{S_0,S_1}"] \arrow[d, "\id"'] && S_{1} \arrow[d, "f"]\\
S_0 \arrow[rr, "\iota_{S_0,S_1'}"'] && S_{1}'
\end{tikzcd}
$$
Then, there is an isomorphism from $g:S_3\rightarrow S_3'$ extending $f\cup\id_{S_2}$, where $S_3'=S_{S_1'\cup S_2}$.
\end{enumerate}
Thus, if one of $S_1$ and $S_2$ is finitely generated, the subsystem $S_{S_1\cup S_2}$ is a co-fibre product of $S_1$ and $S_2$ over $S_1\cap S_2$
\end{lemma}
\begin{proof}
For subsystems $S'\subseteq S'' \subseteq S$, we write $\iota_{S',S''}$ for the inclusion from $S'$ to $S''$. Let $S_0$, $S_1$, and $S_2$ be subsystems of $S$ and let $S_3:=S_{S_1\cup S_2}$. Suppose $\min S_1\neq \emptyset$, $S_0\subset S_1\cap S_2$, and  $S_1\vee S_2\sim \min S_0$.\\

$(1)$ Put $G:=G(S)$. Let $N_i:=\ker(\CG(\iota_{S_i,S}))$ so that we can identify $G(S_i)=G/N_i$ for $i=0,1,2,3$. Let $$X:=\{\CG(\iota_{S_1,S_3})^{-1}[N_1]:N_1\in \CN\left(G(S_1)\right)\}\cup\{\CG(\iota_{S_2,S_3})^{-1}[N_2]:N_2\in \CN\left(G(S_2)\right)\}$$ and let $F_X$ be a pre-sorting data on $X$ given as follows: For $N_1\in \CN(G(S_1))$ and $N_2\in \CN(G(S_2))$, $$F_X(\CG(\iota_{S_1,S_3})^{-1}[N_1]):=F(S_1)(N_1),\ F_X(\CG(\iota_{S_2,S_3})^{-1}[N_2]):=F(S_2)(N_2).$$

\noindent We show that the diagram
$$
\begin{tikzcd}
\CG(S_3)\arrow[rr, "\CG(\iota_{S_2,S_3})"] \arrow[d, "\CG(\iota_{S_1,S_3})"']& & \CG(S_2) \arrow[d, "\CG(\iota_{S_0,S_2})"]\\
\CG(S_1) \arrow[rr, "\CG(\iota_{S_0,S_1})"'] & & \CG(S_0)
\end{tikzcd}
$$
is cartesian, that is,
\begin{itemize}
	\item $G(S_3)$ is the fibre product of $G(S_1)$ and $G(S_2)$ over $G(S_0)$; and
	\item $F(S_3)$ is generated by a pre-sorting data $F_X$.
\end{itemize}

\medskip

Since $S_3=S_{S_1\cup S_2}$ and $S_1\vee S_2\sim \min S_0$, we have $N_3=N_1\cap N_2$ and $N_0=N_1N_2$. Recall that $G/(N_1\cap N_2)$ is the fibre product of $G/N_1$ and $G/N_2$ over $G/N_1N_2$ and so $G(S_3)$ is the fibre product of $G(S_1)$ and $G(S_2)$ over $G(S_0)$. Since $S_3$ is the smallest subsystem containing $S_1$ and $S_2$, the sorting data $F(S_3)$ is generated by the pre-sorting data $F_X$.\\

$(2)$ Let $S_1'$ be a subsystem of $S$ such that $S_0\subset S_1'$ and $S_1'\vee S_2\sim \min S_0$. Suppose there is an isomorphism $f:S_1\rightarrow S_1'$ such that $\iota_{S_0,S_1'}=f\circ \iota_{S_0,S_1}$.\\ 

Let $$X':=\{\CG(\iota_{S_1',S_3'})^{-1}[N_1']:N_1'\in \CN\left(G(S_1')\right)\}\cup\{\CG(\iota_{S_2,S_3'})^{-1}[N_2]:N_2\in \CN\left(G(S_2)\right)\}$$ and let $F_{X'}$ be a pre-sorting data on $X'$ given as follows: For $N_1'\in \CN(G(S_1'))$ and $N_2\in \CN(G(S_2))$, $$F_{X'}(\CG(\iota_{S_1',S_3'})^{-1}[N_1']):=F(S_1')(N_1'),\ F_{X'}(\CG(\iota_{S_2,S_3'})^{-1}[N_2]):=F(S_2)(N_2).$$ Then, the sorting data $F(S_3')$ is generated by $F_{X'}$.\\

We first identify 
\begin{align*}
G(S_3)\cong G(S_1)\times_{G(S_0)}G(S_2)&,\ x\mapsto (\CG(\iota_{S_1,S_3}(x),\CG(\iota_{S_2,S_3})(x)),\\
G(S_3')\cong G(S_1')\times_{G(S_0)}G(S_2)&,\ x'\mapsto (\CG(\iota_{S_1',S_3'}(x'),\CG(\iota_{S_2,S_3'})(x')),
\end{align*}
where
\begin{align*}
G(S_1)\times_{G(S_0)}G(S_2)&:=\{(a,b)\in G(S_1)\times G(S_2):\CG(\iota_{S_0,S_1})(a)=\CG(\iota_{S_0,S_2})(b)\},\\
G(S_1')\times_{G(S_0)}G(S_2)&:=\{(a',b)\in G(S_1')\times G(S_2):\CG(\iota_{S_0,S_1'})(a')=\CG(\iota_{S_0,S_2})(b)\},
\end{align*}
and so we have an isomorphism $$\CG(f)\times \CG(\iota_{S_0,S_2}):G(S_3')\rightarrow G(S_3),\ (a',b)\mapsto (\CG(f)(a'),b).$$

Since for each $N_1'\in \CN(G(S_1'))$, $$F(S_1')(N_1')=F(S_1)(\CG(f)^{-1}[N_1'])$$ and so the isomorphism $\CG(f)\times \CG(\iota_{S_0,S_2}):G(S_3')\rightarrow G(S_3)$ induces an isomorphism from $\CG(S_3')$ to $\CG(S_3)$. Then, the isomorphism $\mathcal S(\CG(f)\times \CG(\iota_{S_0,S_2})):S_3\rightarrow S_3'$ is the desired one. 
\end{proof}

%
%

\begin{remark}\label{rem:co-fibre product_and_smallestsubsystem}
For any subsystems $S_1$ and $S_2$ of $S$, the subsystem $S_{S_1\cup S_2}$ is a co-fibre product of $S_1$ and $S_2$ over $S_0$, where $S_0$ is the subsystem generated by $(S_1\vee S_2)\cap S_1\cap S_2$.
\end{remark}

Next, we introduce a notion of the co-sorted embedding property, which is a dual notion of SEP for sorted profinite groups. Let $S$ be a sorted complete system.
\begin{definition}\label{def:co-(F)SIP}
We say that $S$ has {\em co-sorted embedding property} (co-SEP) if for any finitely generated subsystems $S_1,S_2\subseteq S$, and for any embedding $\Pi:S_2\rightarrow S_1$ and any embedding $\Phi:S_2\rightarrow S_2$, there is an embedding $\Psi:S_1\rightarrow S$ such that 
$$
\begin{tikzcd}
S_2 \arrow[r, "\Pi"] \arrow[d, "\Phi"'] & S_1 \arrow[d, dashed, "\Psi"]\\
S_2 \arrow[r, "\iota"']& S
\end{tikzcd},
$$ where $\iota$ is the inclusion.
\end{definition}

\begin{remark}\label{rem:axiomatizability_co-SIP}
By Remark \ref{rem:smallest subsystem}, co-SEP is first order axiomatizable in the language $\CL_{SCS}(\CJ)$. Denote the theory of sorted complete systems having co-SEP by $SCS_{SEP}$.
\end{remark}
\begin{proof}
For each finitely generated sorted complete systems $S_1$ and $S_2$ and for an embedding $\Pi:S_2\rightarrow S_1$, fix nice finite presystems $X_1$ and $X_2$ generating $S_1$ and $S_2$ respectively such that $\Pi[X_2]\subseteq X_1$. Consider a sentence $\varphi_{S_1,S_2,\Pi}$ given by
$$\forall \bar x_1\bar x_2\left(\theta_{X_1}(\bar x_1)\wedge\theta_{X_2}(\bar X_2)\rightarrow \bigwedge_{\Phi}\exists \bar x_3\left(\theta_{X_1}(\bar x_3)\wedge \bar x_2\subseteq \bar x_3\wedge\bigvee_{\Psi}(\Psi\circ \Pi)(\bar x_2)=\Phi(\bar x_2) \right) \right),$$ where $\Phi:\bar x_2\rightarrow \bar x_2$, and $\Psi:\bar x_1\rightarrow \bar x_3$ are $\CL_{SCS}(\CJ)$-embeddings. Then, the theory of sorted complete systems having co-SEP is given by 
$$
SCS\cup\{\varphi_{S_1,S_2,\Pi}:\ S_1,S_2\mbox{ are finitely generated},\ \Pi:S_2\rightarrow S_1\}.
$$
\end{proof}

\noindent Since the category $\SPG$ of sorted profinite groups and the category of sorted complete systems are equivalent by contravariant functors, we have the following relationship between SEP and co-SEP.
\begin{remark}\label{rem:(F)SIP_co-(F)SIP}
Let $(G,F)$ be a sorted profinite group and let $S$ be the sorted complete system of $(G,F)$. By Theorem \ref{thm:FSIP=SIP=IP+homogeneous_sorting_data}, we have that $(G,F)$ has SEP if and only if $(G,F)$ has FSEP if and only if $S$ has co-SEP.
\end{remark}
For a sorted complete system $S$, let $\coFSIM(S)$ be the set of isomorphism classes of finitely generated subsystems of $S$. We have  the following analogue to \cite[Theorem 2.2]{C1}.
\begin{lemma}\label{lem:key_lemma_for_alpeh_0_categoricity_for_SIP}
Let $S_1$ and $S_2$ be sorted complete systems having co-SEP. Suppose $|S_1|=|S_2|=\aleph_0$ and $\coFSIM(S_1)=\coFSIM(S_2)$. Then, $S_1\cong S_2$.
\end{lemma}
\begin{proof}
We follow the proof scheme of back-and-forth argument in the proof of \cite[Theorem 2.2]{C1}. List $S_1=\{\alpha_0,\alpha_1,\ldots\}$ and $S_2=\{\beta_0,\beta_1,\ldots\}$.

Inductively, we will construct an increasing sequence of isomorphisms $f_i:S_{1,i}\rightarrow S_{2,i}$ between finitely generated subsystems $S_{1,i}$ and $S_{2,i}$ of $S_1$ and $S_2$ respectively, such that for $i\in \omega$, $\alpha_i\in S_{1,i}$ if $i$ is even, and $\beta_i\in S_{2,i}$ if $i$ is odd.  For each $l=1,2$, let $S_{l,-1}$ be the trivial subsystem of $S_l$ such that $m(k,J)(S_{l,-1})$ consists of the maximal element with respect to the partial order $\le$ on $S_l$ for each $(k,J)\in \BN\times \CJ^{<\BN}$. Let $f_{-1}:S_{1,-1}\rightarrow S_{2,-1}$ be the canonical isomorphism. Suppose that we have constructed $f_0,\ldots,f_i$. Without loss of generality, we may assume that $i$ is odd. Suppose $S_{1,i}$ is generated by a finite presystem $X_i$ of $S_1$.

If $\alpha_{i+1}\in S_{1,i}$, put $S_{1,i+1}:=S_{1,i}$ and put $f_{i+1}:=f_i$. If $\alpha_{i+1}\notin S_{1,i}$, let $S_{1,i+1}$ be the subsystem generated by the finite subset $X_i\cup \left([\alpha_{i+1}]\cap m(k,J)\right)$ where $\alpha_{i+1}\in m(k,J)$. Since $\coFSIM(S_1)=\coFSIM(S_2)$, there is a subsystem $\bar S_{2,i+1}$ of $S_2$ which is isomorphic to $S_{1,i+1}$. Let $\bar \psi :S_{1,i+1}\rightarrow \bar S_{2,i+1}$ be an isomorphism. Note that $\bar S_{2,i+1}$ is also finitely generated because $S_{1,i+1}$ is finitely generated. Since $S_2$ has co-SEP, there is an embedding $\psi :\bar S_{2,i+1}\rightarrow S_2$ to make the following diagram commute:
$$
\begin{tikzcd}
S_{2,i}\arrow[r, "f_{i,-1}"] \arrow[d, "\id"'] & S_{1,i+1} \arrow[r, "\bar\psi"] & \bar S_{2,i+1} \arrow[d, dashed, "\psi"]\\
S_{2,i} \arrow[rr, "\iota"'] & & S_2
\end{tikzcd}
$$
Put $S_{2,i+1}:=\psi\circ \bar\psi[S_{1,i+1}]$ and put $f_{i+1}=\psi\circ \bar\psi$. Note that $f_{i+1}$ extends $f_i$ because of the following diagram: 
$$
\begin{tikzcd}
S_{2,i}\arrow[r, "f_{i}^{-1}"] \arrow[d, "\id"'] & S_{1,i+1} \arrow[d, "f_{i+1}"]\\
S_{2,i} \arrow[r, "\iota"'] & S_{2,i+1}
\end{tikzcd}
$$
\end{proof}
\subsection{$\omega$-stability of sorted complete system with co-SEP}
From now on, we assume that $\CJ$ is {\bf countable}.

\begin{remark}\label{rem:firstorderproperty_SIM}
Let $S'$ be a finitely generated sorted complete system and let $X$ be a nice finite presystem generating $S'$. Then, by Remark \ref{rem:smallest subsystem}(4), $S'\in \coFSIM(S)$ if and only if $S\models \exists \bar x\theta_X(\bar x)$.
\end{remark}

\noindent As a corollary of Lemma \ref{lem:key_lemma_for_alpeh_0_categoricity_for_SIP}, we have the following result analogous to \cite[Theorem 2.3]{C1}. 

\begin{theorem}\label{thm:alpeh_0_categoricity_for_SIP}
Suppose that $S$ has co-SEP. For each finitely generated sorted complete system $S'$, fix a nice finite presystem $X_{S'}$ generating $S'$. Then, $Th(S)$ is axiomatized by 
\begin{align*}
&SCS_{SEP}\\
\cup&\{\exists \bar x \theta_{X_{S'}}(\bar x):S'\in \coFSIM(S) \}\\
\cup&\{\neg \exists \bar x\theta_{X_{S'}}(\bar x):S'\notin \coFSIM(S)\}.
\end{align*}
\end{theorem}

\begin{proof}
Let $S_1$ and $S_2$ be sorted complete systems. Suppose
\begin{align*}
S_1,S_2\models &SCS_{SEP}\\
\cup&\{\exists \bar x \theta_{X_{S'}}(\bar x):S'\in \coFSIM(S) \}\\
\cup&\{\neg \exists \bar x\theta_{X_{S'}}(\bar x):S'\notin \coFSIM(S)\}.
\end{align*}
We show that $S_1\equiv S_2$. By Remark \ref{rem:firstorderproperty_SIM}, $\coFSIM(S_1)=\coFSIM(S_2)$ and by L\"owenheim-Skolem, we may assume that $S_1$ and $S_2$ are countable. Then, by Lemma \ref{lem:key_lemma_for_alpeh_0_categoricity_for_SIP}, $S_1\equiv S_2$. 
\end{proof}

Next, we describe compete types of sorted complete systems with co-SEP, which is analogous to \cite[Theorem 2.4]{C1}. 
\begin{definition}
For $a\in m(k,J)$ and for a subsystem $S'$ of $S$, we write $a\vee S':=\{c\in m(kk',J_{\subseteq}^*)(k,J):[c]=[a]\vee [b], b\in \min S'\cap m(k',J')\}$. 
\end{definition}
\noindent Note that $a\vee S'$ is locally full, and $c_1\sim c_2$ for any $c_1,c_2\in a\vee S'$ so that $a\vee S'$ is relatively dense.

\begin{theorem}\label{thm:description_complete_types}
Let $S\models SCS_{SEP}$.  
\begin{enumerate}
	\item Let $A$ be a subsystem of $S$ and let $a,b\in m(k,J)$. Then, $\tp(a/A)=\tp(b/A)$ if and only if $a\vee A=b\vee A$ and there is an isomorphism $f:S_{(a\vee A)\cup\{a\}}\rightarrow S_{(b\vee A)\cup \{b\}}$ such that
	\begin{itemize}
		\item $f\restriction_{S_{a\vee A}}=\id_{S_{a\vee A}}$;
		\item $f(a)=b$.
	\end{itemize}
	\item $\Th(S)$ is $\omega$-stable.
\end{enumerate}
\end{theorem}
\begin{proof}
We follow the proof scheme of \cite[Theorem 2.4]{C1}.

$(1)$ Suppose $\tp(a/A)=\tp(b/A)$. Without loss of generality, we may assume that $S$ is saturated. Then, there is an automorphism $f$ of $S$ over $A$ sending $a$ to $b$. By the choice of $f$, we have that $a\vee A=b\vee A$. Also, the restriction of $f$ on $S_{(a\vee A)\cup \{a\}}$ is the desired one.

Conversely, suppose $a\vee A=b\vee A$ and there is an isomorphism $f:S_{(a\vee A)\cup \{a\}}\rightarrow S_{(b\vee A)\cup \{b\}}$ over $a\vee A$ sending $a$ to $b$. Take a locally full, relatively dense, and finite subset $X$ of $A$ so that $X$ contains an element $x_0\in \min A$. Take a locally full finite subset $Y$ of $a\vee A$. By the choice of $x_0$ and $Y$, we have that 
\begin{itemize}
	\item $S_{X\cup Y}\vee S_{Y\cup \{a\}}\sim \{x_0\}\vee \{a\}\sim S_{a\vee A}\sim S_Y$;
	\item $S_{X\cup Y}\vee S_{Y\cup \{b\}}\sim \{x_0\}\vee \{b\}\sim S_{b\vee A}\sim S_Y$.
\end{itemize}
And the restriction $f\restriction_{S_{Y\cup \{a\}}}:S_{Y\cup \{a\}}\rightarrow S_{Y\cup\{b\}}$ is an isomorphism over $S_Y$ with $f([a]\cap m(k,J))=[b]\cap m(k,J)$ so that we have the following diagram 
$$
\begin{tikzcd}
S_Y \arrow[rr, "\iota_{S_Y,S_{Y\cup \{a\}}}"] \arrow[d, "\id"'] && S_{Y\cup \{a\}} \arrow[d, "f\restriction_{S_{Y\cup \{a\}} }"]\\
S_Y \arrow[rr, "\iota_{S_Y,S_{Y\cup \{b\}}}"'] && S_{Y\cup \{b\}}
\end{tikzcd}
$$
By Lemma \ref{lem:fibre_prod_in_sortedcompletesystem}, there is an isomorphism $g:S_{X\cup Y\cup \{a\}}\rightarrow S_{X\cup Y\cup \{b\}}$ over $S_{X\cup Y}$. Take a countable elementary substructure $S'$ of $S$ containing $X,Y,a,b$ so that $S_{X\cup Y\cup\{a\}}$ and $S_{X\cup Y\cup\{b\}}$ are subsystems of $S'$. By the back-and-forth argument with co-SEP, we extend the isomorphism $g$ into an automorphism of $S'$ and we have that $\tp_{S'}(a/S_X)=\tp_{S'}(b/S_X)$ where $\tp_{S'}(a/S_X)$ is the complete type of $a$ over $S_X$ in $S'$. Since $S'$ is an elementary substructure of $S$, we have that $\tp(a/S_X)=\tp(b/S_X)$, and so by compactness, $\tp(a/A)=\tp(b/A)$.\\

$(2)$ As noted in the proof of \cite[Theorem 2.4]{C1}, it is enough to show that for a countable subset $A$ of $S$, there are only countably many unary types over $A$.

Let $A$ be a countable subset of $S$. Without loss of generality, we may assume that $A$ is a subsystem of $S$ and $S$ is $\aleph_1$-saturated. By $(1)$, for each $a\in m(k,J)$, $\tp(a/A)$ is determined by the isomorphism classes of $S_{a\vee A\cup \{a\}}$ over $S_{a\vee A}$. Since $A$ is countable and $\CJ$ is countable, there are only countably many possibilities of $a\vee A$ for each $a\in m(k,J)$. Also, $a\vee A\cup ([a]\cap m(k,J))$ is locally full and dense in $S_{a\vee A\cup \{a\}}$ because $a\le s$ for all $s\in S_{a\vee A\cup \{a\}}$. Thus, the isomorphism from $S_{a\vee A\cup\{a\}}$ is determined by the image of $a\vee A\cup ([a]\cap m(k,J))$.  Thus, for each $a\in m(k,J)$, the number of the isomorphism classes of $S_{a\vee A\cup \{a\}}$ over $a\vee A$ is bounded by the number of isomorphism classes of $[a]$, which is finite. Therefore, there are only countably many types $\tp(a/A)$ for each $a\in m(k,J)$.
\end{proof}

\subsection{Description of forking} In this subsection, we aim to describe forking independence and $U$-rank in sorted complete systems having co-SEP. We basically follow the proof scheme in \cite[Section 4]{C1}. We fix a complete extension $T$ of $SCS_{SEP}$ and assume that $S\models T$.

\begin{definition}\cite[Definition 1.10]{C1}
Let $a\le b\in S$. The length of $a$ over $b$, denoted by $L(a/b)$ is the largest integer $n$ such that there exists a chain $$a=a_0<a_1<\cdots<a_n=b.$$
\end{definition}

\begin{lemma}\label{lem:join_algebraic_closure}\cite[Lemma 4.1]{C1}
Let $a<b\in S$ and let $A\subset S$. Suppose $b\in a\vee A$. Then, $$a\in \acl(A)\Leftrightarrow a\in \acl(b).$$
\end{lemma}
\begin{proof}
If $b\in \min A$, then it holds trivially. So, we assume that $\{b\}>\min A$. It is enough to show the left-to-right implication holds. Suppose $a\in \acl(A)$. We use induction on $L(a/b)$. Suppose $L(a/b)=1$. Without loss of generality, we may assume that $S$ is $(\aleph_0+|A|)^+$-saturated.

Suppose $a\in \acl(A)\setminus\acl(b)$. Then, $\tp(a/\acl(b))$ has $(\aleph_0+|A|)^+$-many realizations in $S$ and so there are infinitely many realizations $a_0,a_1,\ldots$ of $\tp(a/\acl(b))$ outside of $A$. Since $a_i\models \tp(a/\acl(b))$, we have that $L(a_i/b)=1$. Note that $\{a_i\},\min A<\{b\}$ and $L(a_i/b)=1$ for each $i=0,1,\ldots$. So, for each $i=0,1,\ldots$, $a_i\vee A\le \{b\}$. If $a_i\vee A<\{b\}$, then $b\notin a\vee A$, which contradicts the assumption that $b\in a\vee A$. So, $b\in a_i\vee A$ and $a\vee A=a_i\vee A=[b]$ for all $i$. Because $b\le x$ for all $x\in S_{a\vee A}$, $S_{a\vee A}\subseteq \acl(b)$. Since $a_i\models \tp(a/\acl(b))$, $a_i\models \tp(a/S_{a\vee A})=\tp(a/S_{a_i\vee A})$. By Theorem \ref{thm:description_complete_types}(1), there is an $\CL_{SCS}(\CJ)$-isomorphism from $S_{(a_i\vee A)\cup\{a_i\}}$ to $S_{(a\vee A)\cup\{a\}}$ sending $a_i$ to $a$ and fixing $S_{a\vee A}=S_{a_i\vee A}$ pointwise. In summary, we have that
\begin{itemize}
	\item $a_i\vee A=a\vee A$; and
	\item there is an $\CL_{SCS}(\CJ)$-isomorphism $S_{a_i\vee A\cup\{a_i\}}$ to $S_{a\vee A\cup\{a\}}$ sending $a_i$ to $a$ and fixing $S_{a\vee A}=S_{a_i\vee A}$ pointwise.
\end{itemize}
Therefore, by Theorem \ref{thm:description_complete_types}(1), $a_i\models \tp(a/A)$. Since $a_i$'s are distinct, $a\not\in \acl(A)$, a contradiction. 

Suppose $L(a/b)=n\ge 2$. Take $a=a_0<a_1<\ldots<a_n=b$. Note that each $a_i$ is in $\acl(A)$ because $a\in \acl(A)$. Since $\{b\}>\min A$, we have that $a_{i+1}\in a_i\vee S_{A\cup\{a_{i+1}\}}$ for each $i$. Since $L(a_i/a_{i+1})=1$, by induction, $a_i\in \acl(a_{i+1})$ for each $i$. Thus, $a=a_0\in \acl(a_n)=\acl(b)$.
\end{proof}

We describe forking independence, analogous to \cite[Proposition 4.1]{C1}.
\begin{proposition}\label{prop:forking_description}
Let $A\subseteq B$ be substructures of $S$ and let $a\in S$. Then,
$$a\indo_A B\Leftrightarrow a\vee B\subseteq \acl(A)\Leftrightarrow a\vee B\subseteq \acl(a\vee A).$$
\end{proposition}
\begin{proof}
Without loss of generality, we may assume that $A$ and $B$ are algebraically closed. Since $\min A\ge \min B$, we have that $a\vee A\ge a\vee B$. Since $\{a\}, \min B\le a\vee B$, we have that $a\vee B\subseteq \acl(a)\cap \acl(B)$. So, if $a\vee B\not\subseteq \acl(A)$, then $B\depo_A a$. Namely, if $B\indo_A a$, then $\acl(B)\indo_{\acl(A)}\acl(a)$ and so $a\vee B\indo_{\acl(A)}a\vee B$ by monotonicity of non-forking and the fact that $a\vee B\subseteq \acl(B)\cap \acl(a)$, which implies $a\vee B\subseteq \acl(A)$, a contradiction. Suppose $a\vee B\subseteq \acl(A)=A$ so that $a\vee B=a\vee A$ and $S_{a\vee B}\subseteq A$. Consider the following partial type over $B$, $$\Sigma(x):=\tp(a/S_{a\vee B})\cup\{x\vee \delta\sim c:\delta\in B,c\in a\vee B, \delta\le c\}.$$ By Theorem \ref{thm:description_complete_types}(1), the partial type $\Sigma$ is consistent and $\Sigma\models \tp(a/B)$. Since $a\vee B=a\vee A\subseteq A$, the type $\tp(a/B)$ is definable over $A$ and it does not fork over $A$. The second equivalence comes from Lemma \ref{lem:join_algebraic_closure}.
\end{proof}
\noindent By the same proof of \cite[Theorem 4.2]{C1}, we have the following description of $U$-rank.
\begin{theorem}\label{thm:U-rank}
Let $a\in S$ and $A\subset S$. Let $n=L(a/b)$ for some (equivalently, any) $b\in a\vee S_A$. Choose a sequence $a=a_0<a_1<\ldots<a_n=b$. Then, the $U$-rank of $\tp(a/A)$ is the number of indices $i<n$ such that $a_i\not\in \acl(a_{i+1})$.
\end{theorem}

We end our paper with the following question.
\begin{question}\label{question:sorted_prof_gp=galois_gp?}
We know by \cite[Theorem 2]{W74} that any profinite group is the Galois group of some field extension, that is, for a given profinite group $G$, there is a Galois extension $L$ of $K$ whose Galois group is $G$. More generally, Hoffmann in \cite[Corollary 3.3]{H20} showed that given a stable theory $T$ eliminating quantifiers and imaginaries, any profinite group is the Galois group of a Galois extension in a monster model, that is, for a given profinite group $G$, there is a Galois extension $K\subseteq L\subset \mathfrak{C}$ with $G(L/K)\cong G$ where $\mathfrak{C}$ is a monster model of $T$.

We ask whether the following generalisation holds. Fix a language $\CL$ with a set $\CJ$ of all sorts with the functions $J^*_{\cap}$ and  $J^*_{\subseteq}$ given in Example \ref{ex:typical_example_galois_gp}. Is there a stable $\CL$-theory eliminating quantifiers and imaginaries such that for a monster model $\mathfrak{C}$ of $T$, any sorted profinite group $(G,F)$ is a Galois group of a Galois extension $K\subseteq L\subset \mathfrak{C}$, that is, for the Galois group $(G(L/K),F_{G(L/K)})$ defined in Example \ref{ex:typical_example_galois_gp}, there is a isomorphism $\varphi:(G,F)\rightarrow (G(L/K),F_{G(L/K)})$ such that $F_{G(L/K)}=\varphi_*(F)$? 
\end{question}

\end{document}